\documentclass[12pt, a4paper]{amsart}
\usepackage[english]{babel}
\usepackage{amsmath,enumerate, amsfonts, amssymb,amsthm, xypic}
\usepackage[dvips]{graphics}
\usepackage{color}
\usepackage{graphicx}
\usepackage{hyperref}

\newtheorem{thm}{Theorem}[section]
\newtheorem{lem}[thm]{Lemma}
\newtheorem{prop}[thm]{Proposition}
\newtheorem{cor}[thm]{Corollary}

\theoremstyle{definition}
\newtheorem{defi}[thm]{Definition}

\newtheorem{ex}[thm]{Example}
\newtheorem{rem}[thm]{Remark}

\DeclareMathOperator{\R}{\mathbb R}
\DeclareMathOperator{\PP}{\mathbb P}

\DeclareMathOperator{\N}{\mathcal N}

\DeclareMathOperator{\E}{\mathcal E}
\DeclareMathOperator{\ES}{\mathcal E_S}
\DeclareMathOperator{\Z}{\mathbb Z}
\DeclareMathOperator{\LL}{\mathbb L}
\DeclareMathOperator{\Sph}{\mathbf S}

\DeclareMathOperator{\id}{id}

\DeclareMathOperator{\lk}{lk}

\DeclareMathOperator{\Var}{Var}
\DeclareMathOperator{\SA}{SA}
\DeclareMathOperator{\AS}{AS}
\DeclareMathOperator{\Reg}{Reg}

\DeclareMathOperator{\jac}{jac}
\DeclareMathOperator{\mult}{mult}
\DeclareMathOperator{\eps}{\epsilon}

\DeclareMathOperator{\D}{\mathcal D}
\DeclareMathOperator{\DS}{\mathcal D_S}\DeclareMathOperator{\DT}{\mathcal D_T}
\DeclareMathOperator{\DAS}{\mathcal D^{\AS}}
\DeclareMathOperator{\DSAS}{\mathcal D^{\SA}_S}
\DeclareMathOperator{\DNS}{\mathcal D^{\N}_S}
\DeclareMathOperator{\LSAS}{\Lambda ^{\SA}_S}

\DeclareMathOperator{\M}{\mathcal M}
\DeclareMathOperator{\MAS}{\mathcal M^{AS}}
\DeclareMathOperator{\MASS}{\mathcal M^{AS}_S}
\DeclareMathOperator{\MNS}{\mathcal M^{\N}_{S}}
\DeclareMathOperator{\MS}{\mathcal M_{S}}

\DeclareMathOperator{\K}{K_0(\R Var)}
\DeclareMathOperator{\KS}{K_0(\R Var_{S})}
\DeclareMathOperator{\KblS}{K_0^{bl}(\R Var_{S})}
\DeclareMathOperator{\KSA}{K_0(SA)}
\DeclareMathOperator{\KregS}{K_0(Const_S)}
\DeclareMathOperator{\KSAS}{K_0(SA_S)}
\DeclareMathOperator{\KSAT}{K_0(SA_T)}
\DeclareMathOperator{\KAS}{K_0(AS)}
\DeclareMathOperator{\KASS}{K_0(AS_S)}

\DeclareMathOperator{\KN}{K_0^{bl}(\mathcal N)}
\DeclareMathOperator{\KNS}{K_0^{bl}(\mathcal N_{S})}
\DeclareMathOperator{\bAS}{\beta^{AS}}
\DeclareMathOperator{\LS}{\lambda_S}
\DeclareMathOperator{\Ls}{\Lambda_s}

\DeclareMathOperator{\un}{\mathbf 1}
\DeclareMathOperator{\tEI}{\widetilde E^+_I}
\DeclareMathOperator{\tEJ}{\widetilde E^+_J}
\DeclareMathOperator{\tE0I}{\widetilde E^{0,+}_I}

\title{On relative Grothendieck rings and algebraically constructible functions}
\author{Goulwen Fichou}
\thanks{The author wish to thank Jean-Baptiste Campesato, Michel Coste, Toshizumi Fukui and Adam Parusi\'nski for useful discussions, and is deeply grateful to the UMI PIMS of the CNRS where this project has been carried out. He has also received support from ANR-15-CE40-0008 (D\'efig\'eo).}
\address{IRMAR (UMR 6625), Universit\'e de Rennes 1, Campus de Beaulieu, 35042 Rennes Cedex, France}

\date\today
\subjclass[2010]{}

\begin{document}

\begin{abstract} We investigate Grothendieck rings appearing in real geometry, notably for arc-symmetric sets, and focus on the relative case in analogy with the properties of the ring of algebraically constructible functions defined by McCrory and Parusi\'nski. We study in particular the duality and link operators, including its behaviour with respect to motivic Milnor fibres with signs.
\end{abstract}
\maketitle
Let $S\subset \R^n$ be a semialgebraic set. A semialgebraically constructible function on $S$ is an integer valued function that can be written as a finite sum $\sum_{i\in I} m_i \un_{S_i}$ where for each $i\in I$, $m_i$ is an integer and $\un_{S_i}$ is the characteristic function of a semialgebraic subset $S_i$ of $S$. The sum and product of two semialgebraically constructible functions are again semialgebraically constructible, so that they form a commutative ring, denoted by $F(S)$. 
An important tool for the study of the ring $F(S)$ is the integration along the Euler characteristic \cite{Sch,V}, which can be defined using the Euler characteristic with compact supports as a measure \cite{Co}. It enables to define a push-forward associated with semialgebraic mappings. Two operators of particular interest exist on $F(S)$, the duality and the link, and they are related to local topological properties of semialgebraic sets.

Dealing with real algebraic sets rather than semialgebraic ones, the push-forward along a regular mapping of the characteristic function of a real algebraic set can not be expressed in general as a linear combination of characteristic functions of real algebraic sets. Nevertheless, the set of all such push-forwards forms a subring $A(S)$ of $F(S)$, which is endowed with the same operations. That ring $A(S)$ of algebraically constructible functions, introduced by McCrory and Parusi\'nski in \cite{MCP}, gives rise to a huge number of invariants for semi-algebraic sets to be locally homeomorphic to real algebraic sets. Another interesting ring in between $F(S)$ and $A(S)$ is the ring $N(S)$ of 
Nash constructible functions, related to the study of the arc-symmetric sets of Kurdyka \cite{K}.

\vskip 5mm

Cluckers and Loeser noticed in the introduction of \cite{CL} that $F(S)$ is isomorphic to the relative Grothendieck ring of semialgebraic sets over $S$, the push-forward corresponding to the composition with a semialgebraic mapping (cf. Proposition \ref{K0-const}). Our aim is the paper is to continue the analogy further in order to relate the rings of algebraically constructible functions and Nash constructible functions to Grothendieck rings appearing in real geometry. Since the development of the theory of motivic integration by Denef and Loeser \cite{DL}, various Grothendieck rings have been considered in real geometry: in the semialgebraic setting, where the corresponding Grothendieck ring happens to be isomorphic to $\Z$ (\cite{Q} and \cite{KoP} for motivic zeta functions in that context), in the algebraic setting \cite{MCPvirt,F,C}, sometimes with a group action \cite{Pr}, and for arc-symmetric sets \cite{C}. We focus in the present paper on Grothendieck rings relative to a base variety $S$, meaning that the building blocks consist of mapping $h:X\to S$ in the corresponding categories. Such relative Grothendieck rings have not been under investigating in the real context yet (except in \cite{C} on the base variety $\R^*$ in order to obtain a Thom-Sebastiani type formula), whereas the motivic zeta functions arriving in the local study already give rise to many interesting invariants (cf. \cite{C3} for a recent survey on the subject). 

\vskip 5mm
In this paper, we study several Grothendieck rings in real geometry, defined with respect to different classes of functions appearing naturally in that context (semialgebraic, arc-analytic \cite{K}, rational continuous \cite{K,Ko}, Nash \cite{Shiota}, and regular functions). We prove in particular that the Grothendieck ring of rational continuous functions is isomorphic to that of real algebraic varieties (Proposition \ref{regu}). We determine also completely the Grothendieck ring of arc-symmetric sets (Theorem \ref{thm-AS}), describing in term of piecewise homeomorphism the equality of two classes (Theorem \ref{LL}).

Assume $S$ is a real algebraic variety.
We show that the relative Grothendieck ring $\KS$ of real algebraic varieties over $S$ maps surjectively to the ring $A(S)$ of algebraically constructible functions, together with an analogous statement for the ring $N(S)$ (Theorem \ref{prop-alg}). Passing to a localisation $\MS$ of $\KS$, we define a duality operator on $\MS$ using the approach developed by Bittner in \cite{B}. This duality shares many properties with the duality defined on the ring of constructible functions, as developed in section \ref{sect-dual} of the present paper. We prove in particular that the duality acts on the motivic Milnor fibres associated with a real polynomial function as if it were the class of a nonsingular variety proper over $S$ (Theorem \ref{dualM}). Along the way, we produce an elementary formula for the motivic Milnor fibres with sign associated with a non-degenerate weighted homogeneous polynomials (Proposition \ref{propM}). 

We introduce also an operator on $\MS$ sharing some properties with the link on constructible functions, but overall we define a local analogue for the link operator (cf. Theorem \ref{thm-l}) when we pass to the virtual Poincar\'e polynomial \cite{MCPvirt}, an algebraic strengthening of the Euler characteristic with compact supports.

\section{Grothendieck rings in real geometry}\label{sect-Gro}

A large class of functions play a role in real algebraic geometry, from semialgebraic functions to regular or polynomial functions, passing to arc-analytic, rational continuous or Nash functions. We consider in this section the different Grothendieck rings associated with these classes of functions, establishing some relations between them. We have a particular focus in this paper on Grothendieck rings relative to a base variety, as introduced in Looijenga \cite{looi} for motivic integration. That point of view has not been much considered in real geometry, except in the approach of Campesato towards a Thom-Sebastiani formula \cite{C}. Our interest in the relative setting is to make some links with the theory of constructible functions in section \ref{sect-cons}.

\subsection{Semialgebraic sets}
A semialgebraic set is a set belonging to the Boolean algebra generated by subsets of some $\R^n$, with $n\in \mathbb N$, defined by polynomial equalities and inequalities. We refer to \cite{BCR} for general background about semialgebraic sets.

The Grothendieck ring $\KSA$ of semialgebraic sets is the quotient of the free abelian group on symbols $[X]$, for each semialgebraic set $X$, by the relation $[X]=[Y]$ if $X$ and $Y$ are semialgebraically homeomorphic, and the scissor relation $[X]=[X\setminus Y]+[Y]$ for $Y\subset X$ a semialgebraic subset of $X$. The ring structure is induced by the product of semialgebraic sets. The ring $\KSA$ have been first studied by Quarez \cite{Q}. It is shown there that the cell decomposition theorem for semialgebraic sets enables to prove that the Grothendieck ring $\KSA$ is isomorphic to the ring of integers via the Euler characteristic with compact supports $\chi_c:\KSA \to \Z$.

\vskip 5mm

For $S$ a semialgebraic set, the relative Grothendieck ring $\KSAS$ associated with a semialgebraic set $S$ is the free abelian group generated by the classes $[h:X\to S]$ of semialgebraic morphism $h:X\to S$, modulo the relations
$$[h:X\to S] = [h_{|X\setminus Y}:X\setminus Y\to S] + [h_{|Y}:Y\to S],$$
where $Y\subset X$ is a semialgebraic subset of $X$. The ring structure is given by fibred products over $S$. Note that we have a $\KSA$-linear mapping $\KSAS \to \KSA$ defined by assigning the value $[X]$ to $[h:X\to S]$.

\subsection{Real algebraic varieties}
The Grothendieck $K_0(\Var_k)$ of algebraic varieties over a based field $k$ has been widely studied, and the study includes the case $k=\R$, where an algebraic variety over $\R$ should be understood in the scheme-theoretic sense. We are interested however in this paper in real algebraic varieties in the sense of \cite{BCR}, meaning that we consider affine real algebraic subsets of some affine space together with the sheaf of regular functions, i.e. rational functions such that the denominators do not vanish on the source space.
For our purpose, it is sufficient to consider only affine varieties since any quasiprojective real variety is affine.

Denote by $\K$ the Grothendieck group of real algebraic varieties, defined as the free abelian group on isomorphism classes of real algebraic varieties, 
modulo the scissor relation $[X] = [X\setminus Y]+[Y]$, where $Y \subset X$ is a closed subvariety.
The group $\K$ carries a ring structure induced by the product of varieties. We denote by $\LL$ the class
of the affine line $[A^1] \in \K$. Let $\M$ be the localization $\K[\LL^{-1}]$. 

\begin{rem} We have a natural morphism $K_0(\Var_{\R}) \to \K$ consisting of taking the real points of an algebraic variety over $\R$. It is not injective since any algebraic variety over $\R$ without real point is sent to zero.
\end{rem}

\begin{rem}\label{rem-const} A Zariski constructible subset of $\R^n$ is a set belonging to the Boolean algebra generated by algebraic subsets of $\R^n$. A Zariski constructible set $C$ is in particular a finite disjoint union $X=\sqcup_i X_i\setminus Y_i$ of locally closed sets, where $Y_i\subset X_i$ is a closed algebraic subset of $X_i$. By the additivity relation, the class
$$[C]:=\sum_i [X_i]-[Y_i]\in \K$$
does not depend on the description of $C$.
\end{rem}

The Euler characteristic with compact supports still gives a realisation of $\K$ in $\Z$, but we know a better realisation via the virtual Poincar\'e polynomial $\beta:\K\to \Z[u]$ defined by McCrory and Parusi\'nski in \cite{MCPvirt}, and whose evaluation at $u=-1$ recovers $\chi_c$.

\begin{thm}(\cite{MCPvirt, MCPI})\label{thm-virt} There exists a ring epimorphism $\beta:\K\to \Z[u]$ characterised by the fact that $\beta([X])$ is equal to the (classical) Poincar\'e polynomial of $X$, namely
$$\beta([X])=\sum_{i=0}^{\dim X}\dim H_i(X,\Z/2\Z) u^i,$$
if $X$ is a compact nonsingular real algebraic variety.
\end{thm}

\begin{rem}\label{rem-surj}
\begin{enumerate}
\item \label{beta-deg} An important property of the virtual Poincar\'e polynomial is that the degree of $\beta(X)$, for a real algebraic variety $X$, is equal to the dimension of $X$.
\item  Note that $\beta$ is surjective. To see this, remark that $\beta(\LL^n)=u^n$ because $\beta(\Sph^n)=1+u^n$ and $\beta$ is additive. Now, if $P(u)=\sum a_iu^i$ is a polynomial in $\Z[u]$, then $P(u)$ is the image under $\beta$ of the element $P(\LL)=\sum a_i \LL^i$ of $\K$.
\item Using Poincar\'e duality, note that for a compact nonsingular real algebraic variety $X$, the virtual Poincar\'e polynomial of $[X]$ evaluated at $1/u$ is equal to $u^{-\dim X} \beta([X])$.
\end{enumerate}
\end{rem}

\vskip 5mm

Let $S$ be a real algebraic variety. By a $S$-variety, we mean a real algebraic variety $X$ together with a regular morphism $X\to S$. The $S$-varieties form a category $\Var_S$, where the arrows are given by those morphisms which commute with the morphisms to $S$. We denote by $\KS$ the Grothendieck group of $S$-varieties. As a group, $\KS$ is the
free abelian group on isomorphism classes $[h:X\to S]$ of $S$-varieties $X$, modulo the relations 
$$[h:X\to S] = [h_{|X\setminus Y}:X\setminus Y\to S] + [h_{|Y}:Y\to S],$$
where $Y\subset X$ is a closed subvariety. The ring structure is induced by the fibred product over $S$ of $S$-varieties.
Let $\MS$ denote the localization $\KS[\LL^{-1}]$.

A morphism $f:S \to T$ induces a ring morphism $f^*: \M_T \to \MS$ by pulling back, making $\M_T$ as a $\MS$-module. We have also a morphism of $\M_T$-modules $f_!: \MS \to \M_T$ by composition. 

\begin{ex}
\begin{enumerate}
\item When $S$ is reduced to a point, the ring $\KS$ is isomorphic to $\K$. 
\item Considering the projection $f:S\to \{p\}$ onto a point, we obtain a structure of $\M$-module for $\MS$ via $f^*$, whereas the induced mapping $f_!:\MS \to \M$ which maps $[h:X\to S]$ onto $[X]$ is $\M$-linear.
\end{enumerate}
\end{ex}

\begin{rem}\label{rem-const2} Similarly to the absolute case in Remark \ref{rem-const}, we can consider the class of Zariski constructible sets in $\KS$.
\end{rem}

We know from \cite{MCPvirt} in the absolute real case, and Bittner \cite{B} in the relative case for algebraic varieties in characteristic zero, that the ring $\KS$ admits a simpler presentation involving only proper mappings from nonsingular varieties. We state it below since we use it later on.
Denote by $\KblS$ the free abelian group generated by isomorphism classes of proper mapping $h:X\to S$, with $X$ nonsingular, subject to the relations $[\emptyset \to S]=0$ and 
$$[h\circ \pi: Bl_C X\to S]-[(h\circ \pi)_{|E}: E\to S]=[h:X\to S]-[h_{|C}:C\to S]$$
where $C\subset X$ is a nonsingular closed subvariety, $\pi: Bl_C X\to X$ is the blowing-up of $X$ along $C$ with exceptional divisor $E$. The ring structure on $\KblS$ is given by the fibred product.

\begin{thm}\label{bit}\cite{B,MCPvirt} The natural mapping $\KblS\to \KS$ is an isomorphism.
\end{thm}

Two real algebraic varieties $X$ and $Y$ are called stably birational if there exist integers $k$ and $l$ such that $X\times \PP^k$ is birational to $Y\times \PP^l$.
As a consequence of Theorem \ref{bit}, together with the Weak Factorisation Theorem \cite{W}, we produce like in \cite{LL} a realisation of $\K$ in the ring $\Z [SB]$, where the ring $\Z [SB]$ is the free abelian group generated by stably birational classes of compact nonsingular irreducible real algebraic varieties, the product of varieties giving the multiplication.

\begin{cor} There exist a ring morphism $\K\to \Z[SB]$ assigning to the class $[X]$ of a compact nonsingular irreducible real algebraic variety $X$ its stably birational class. Its kernel is generated by the class of the affine line.
\end{cor}

This realisation is important to produce some zero divisor in the Grothendieck rings.

\begin{ex}\label{ex-bir} Let $C_1$ be a circle in $\R^2$ and $C_2$ be given by $x^4+y^4=1$ in $\R^2$. Then $[C_1]-[C_2]$ is non-zero in $\K$ (whereas the virtual Poincar\'e polynomial of $C_1$ and $C_2$ coincide). Otherwise $C_1$ and $C_2$ would be stably birational, and therefore birational (stably birational curves are always birational, cf. \cite{D}). Here is a contradiction since these curves do not have the same genus (zero versus three) which is a birational invariant.
\end{ex}

\begin{rem} Following \cite{Poo}, we see also that $\K$ is not as domain, since for smooth compact irreducible algebraic variety over $\R$, the birationality of their set of real points implies their birationality.
\end{rem}

We give finally description of $\KS$ using continuous rational functions.
A continuous rational function on a nonsingular real algebraic variety is a continuous real value function which coincides with a rational function of a Zariski dense open subset of the source variety. This class of functions has been firstly studied by Kucharz \cite{Ku} and Koll\'ar and Nowak \cite{Ko}, and a systematic studied of the smooth affine case can be found in \cite{FHMM}. When the variety is singular, the notion of continuous rational functions is quite tricky, but a rather tame class of functions is given by the so called hereditarily rational functions in \cite{Ko}, or regulous functions in \cite{FHMM}. For simplicity, consider $X\subset \R^n$ being an algebraic subset of some $\R^n$. Then a regulous function on $X$ is the restriction to $X$ of a continuous rational function on $\R^n$.
The crucial fact about regulous functions is that they remains rational in restriction to any algebraic subset of $X$, cf. \cite{Ko}, Theorem 10.
A regulous mapping is a mapping whose components are regulous functions. A regulous isomorphism is a bijective regulous mapping whose inverse is a regulous mapping.

The zero sets defined by regulous functions are the closed Zariski constructible sets (\cite{FHMM}, Theorem 6.4).
For a Zariski constructible set $S$, denote by $\KregS$ the Grothendieck ring generated by regulous isomorphism classes of regulous mapping $h:C\to S$, where $C$ is a Zariski constructible set, subject to the usual additivity and multiplicativity relations.

We noticed already in Remark \ref{rem-const} that the Grothendieck ring of algebraic varieties contains the classes of Zariski constructible sets. We show below that $\KS$ is actually isomorphic to $\KregS$ when $S$ is an algebraic variety. 

\begin{prop}\label{regu} Let $S$ be a real algebraic variety. Then $\KS$ is isomorphic to $\KregS$.
\end{prop}

\begin{proof}
The class of a constructible set is a combination of classes of algebraic sets, cf. Remark \ref{rem-const}. The same is true for the class of a regulous mapping $h:C\to S$ from a constructible set $C$ to $S$, as we can prove as follows, using an induction on the dimension of $C$. Indeed, $h$ is regular except on its indeterminacy locus $D$ which is a Zariski constructible subset of $C$ of dimension strictly smaller than $\dim C$. Moreover $h$ is still a rational function in restriction to $D$ by \cite{Ko}, so that $h_{|D}:D\to S$ is a regulous mapping. Using the induction hypothesis, the restriction $h_{|D}$ can be expressed as a combination of classes of regular mappings, therefore the same is true for $[h:C\to S]$ since
$$[h:C\to S]=[h_{|C\setminus D}:C\setminus D \to S]+[h_{|D}:D \to S].$$
Finally, if $h_1:C_1\to S$ and $h_2:C_2\to S$ are regulous isomorphic regular mappings from locally closed algebraic sets, meaning there exist a regulous isomorphism $\phi:C_1\to C_2$ commuting with $S$, then we can stratify similarly $C_1$ and $C_2$ into locally closed sets $C_1=\sqcup C_{1,i}$ and $C_2=\sqcup C_{2,i}$ such that for all $i$, the rational mapping $\phi$ is regular in restriction to $C_{1,i}$, the image $\phi(C_{1,i})$ is equal to $C_{2,i}$ and $\phi^{-1}$ is regular in restriction to $C_{2,i}$. As a consequence the image in $\KS$ of $h_1:C_1\to S$ and $h_2:C_2\to S$ coincide as expected.
\end{proof}

\subsection{Arc-symmetric sets}
Arc-symmetric sets have been introduced by Kurdyka in \cite{K} as subsets of $\R^n$, for some $n\in \mathbb N$, stable along analytic arcs. We consider in the present paper arc-symmetric sets as subsets of a projective space, as defined by Parusi\'nski in \cite{P}, together with the constructible category $\AS$ of boolean combinations of arc-symmetric sets. 

\begin{defi} A semialgebraic subset $A\subset \PP^n$ is an $\AS$-set if for every real analytic arc $\gamma: (−1,1)\to \PP^n$ such that $\gamma((−1,0))\subset A$, there exists $\epsilon>0$ such that $\gamma((0,\epsilon))\subset A$. An $\AS$-mapping is a mapping whose graph is an $\AS$-set.
\end{defi}

\begin{rem} \begin{enumerate}
\item An $\AS$-set closed in the Euclidean topology is called an arc-symmetric set, and the family of $\AS$-sets corresponds to the boolean algebra generated by arc-symmetric sets.
\item The dimension of an $\AS$-set is its dimension as a semialgebraic set. There exist also a notion of arc-symmetric closure of a semialgebraic set (conserving the dimension). We refer to \cite{KP} for a nice introduction to $\AS$-sets.
\end{enumerate} 
\end{rem}

We denote by $\KAS$ the Grothendieck ring of $\AS$-sets (which has already been considered in \cite{C}). It is generated as a group by the classes of $\AS$-sets under $\AS$-homeomorphism, subject to the scissor relation
$[A]=[B]+[A\setminus B]$ where $B$ is an $\AS$-subset of $A$, the ring structure being induced by the product of $\AS$-sets.

We define also the relative Grothendieck ring $\KASS$ of $\AS$-sets over an $\AS$-set $S$ by considering the classes of $\AS$-mapping $h:A\to S$, where $A$ is an $\AS$-set. It comes with natural morphisms
$$\KS \to \KASS \to \KSAS.$$
Finally, we denote by $\MASS$ the localisation of $\KASS$ in the class of the affine line of $\KAS$.

\begin{rem} Note that we obtain the same Grothendieck ring $\KASS$ by replacing the isomorphism condition of being an $\AS$-homeomorphism by the weaker condition of being an $\AS$-bijection (cf. \cite{MCP} Theorem 4.6, or more explicitly \cite{C} Remark 4.15). Note also that it is not necessary to assume that $B$ is closed in $A$ in the scissor relation.
\end{rem}

Let us recall that the virtual Poincar\'e polynomial can be extended from real algebraic varieties to $\AS$-sets. By a nonsingular $\AS$-set, we mean an $\AS$-set which is included in the nonsingular locus of its Zariski closure.

\begin{thm}(\cite{F}) There exists a ring morphism $\bAS:\KAS \to \Z[u]$ characterised by the fact that $\bAS([X])$ is equal to the (classical) Poincar\'e polynomial of $X$ for $X$ a compact nonsingular arc-symmetric set.
\end{thm}

A real algebraic set is in particular an arc-symmetric set, and we have a ring morphism $\alpha: \K \to \KAS$, satisfying $\bAS \circ \alpha =\beta$.

In the algebraic context, the Grothendick ring of varieties is still a mysterious object, despite of its realisation morphisms. More generally, very few Grothendieck rings appearing in geometric context are well-understood, the semialgebraic case being a rather simple (but non trivial) exception. The following result makes a progress in this direction, providing a full description of $\KAS$ via the virtual Poincar\'e polynomial.

\begin{thm}\label{thm-AS} The Grothendieck ring $\KAS$ of $\AS$-sets is isomorphic to $\Z[u]$ via the ring morphism $\bAS:\KAS \to \Z[u]$ induced by the virtual Poincar\'e polynomial.
\end{thm}

In view of the proof of Theorem \ref{thm-AS}, we begin with expressing the class of an $\AS$-set in term of classes of compact nonsingular real algebraic varieties. 
For an $\AS$-set $A$, we denote by $\Reg A$ its regular part, namely $\Reg A:= A\cap \Reg(\overline {A}^Z)$.

\begin{lem}\label{lem-class} The class of any arc-symmetric $A\in \AS$ in $\KAS$ is equal to a finite linear combination $\sum n_i \alpha([A_i])$, with $n_i\in \Z$, of classes of compact nonsingular real algebraic varieties $A_i$.
\end{lem}

\begin{proof} We prove the result by induction on the dimension of $A$. The result is trivial in dimension zero, so let assume $A$ has positive dimension. Note that it is sufficient to prove the result for an irreducible arc-symmetric set, since the intersection of two different irreducible arc-symmetric sets of the same dimension gives an arc-symmetric set of dimension strictly smaller.

Compactifying the Zariski closure of $A$ if necessary, we may assume that $\overline {A}^Z$ is an compact irreducible real algebraic set. Let ${\overline A}^{\AS}$ denote the arc-symmetric closure of $A$ in $\overline {A}^Z$. Then ${\overline A}^{\AS}$ is compact and ${\overline A}^{\AS}\setminus A$ is an arc-symmetric set of dimension strictly smaller than the dimension of ${\overline A}^{\AS}$, therefore by the induction assumption it suffices to prove the result for ${\overline A}^{\AS}$. Applying the resolution of singularities for arc-symmetric sets (cf. Theorem 2.6 in \cite{K}), there exist a compact nonsingular real algebraic variety $X$, a composition $\pi:X\to \overline {A}^Z$ of blowings-up along nonsingular algebraic centres, and a connected component $C$ of $X$ such that $\pi(C)$ is equal to the Euclidean closure of the regular part $\Reg A$ of $A$. Denote by $D\subset \overline {A}^Z$ and $E\subset X$ the smallest algebraic subsets such that $\pi$ is an isomorphism from $X\setminus E$ onto $ \overline {A}^Z \setminus D$, where the dimension of $D$ and $E$ are dimension strictly smaller than the dimension of $A$. Then we have more precisely 
$$\pi(C\setminus (C\cap E))=A\setminus (A\cap D),$$
hence the equality
$$[A]=[C]-[C\cap E]+[A\cap D]$$
in $\KAS$, since $\pi$ is an isomorphism on restriction to $C \setminus (C\cap E)$. Note that $C$ is a compact connected Nash manifold as a connected component of a compact nonsingular real algebraic set. By Nash-Tognoli Theorem (cf. \cite{BCR} Theorem 14.1.10), $C$ is diffeomorphic to a compact nonsingular real algebraic set $C'$. Approximating the diffeomorphism by a Nash diffeomorphism (meaning a real analytic diffeomorphism with semialgebraic graph, cf. \cite{Shiota}), we see that $C$ is Nash diffeomorphic to the real algebraic set $C'$, so that $[C]=\alpha ([C'])$ in $\KAS$.

Finally, note that the dimension of  $C\cap E$ and $A\cap D$ are strictly smaller than the dimension of $A$, so their class can also be expressed as a finite linear combination of classes of compact nonsingular real algebraic set using the induction hypothesis.
\end{proof}

The proof of Theorem \ref{thm-AS} is, in an essential manner, a consequence of Theorem 4.5 in \cite{MCPI}, combined with Lemma \ref{lem-class}. 

A Nash manifold is a semialgebraic set endowed with the structure of a smooth real analytic variety. A Nash morphism is a real analytic morphism with semialgebraic graph. A Nash set is the zero locus of a Nash function. We refer to \cite{Shiota} for more on Nash manifolds. 

\begin{proof}[Proof of Theorem \ref{thm-AS}]
The ring morphism $\alpha:\K\to \KAS$ satisfies $\alpha([X])=\alpha([Y])$ whenever $X$ and $Y$ are compact nonsingular real algebraic sets which are Nash diffeomorphic. Applying Theorem 4.5 in \cite{MCPI}, there exist $h:\Z[u] \to \KAS$ such that $\alpha=h \circ \beta$. Let prove that $h$ is an inverse for $\bAS$.

By Lemma \ref{lem-class}, any class $[A]$ in $\KAS$ may be expressed as a finite linear combination $\sum n_i \alpha([A_i])$ of classes of compact nonsingular real algebraic varieties $A_i$, with $n_i \in \Z$. Therefore 
$$h \circ \bAS (A)=h \circ \bAS \circ \alpha (\sum n_i [A_i])= h \circ \beta (\sum n_i [A_i])$$
since $\bAS \circ \alpha=\beta$. Moreover $\alpha=h \circ \beta$, therefore
$$h \circ \bAS (A)=\alpha(\sum n_i [A_i])=[A].$$
Now, if $P(u)$ is a polynomial in $\Z[u]$, then $P(u)=\beta(P(\LL))$ by Remark \ref{rem-surj}. Then
$$\bAS \circ h (P(u))=\bAS \circ h \circ \beta (P(\LL))=\bAS \circ \alpha (P(\LL))$$
since $h\circ \beta=\alpha$. Moreover $\bAS \circ \alpha=\beta$ so that 
$$\bAS \circ h (P(u))=\beta(P(\LL))=P(u)$$
as required.
\end{proof}

The surjectivity of $\beta:\K \to \Z[u]$ induces the following corollary.

\begin{cor} The morphism $\alpha:\K\to \KAS$ is surjective.
\end{cor}

\begin{rem} Note however that $\alpha$ is not injective, by Example \ref{ex-bir}.
\end{rem}

We can deduce from Theorem \ref{thm-AS} another presentation of $\KAS$, in the spirit of Theorem \ref{bit}. 
Denote by $\KN$ the free abelian group generated by Nash isomorphism classes of compact Nash manifolds, subject to the relations $[\emptyset]=0$ and  $[Bl_C X]-[E]=[X]-[C]$, where $X$ is a compact Nash manifold, $C\subset X$ is a Nash submanifold, $Bl_C X$ is the blowing-up of $X$ along $C$ and $E$ is the exceptional divisor. The ring structure is given by taking fibred products.

\begin{cor}\label{cor-AS} The natural morphism $\KN\to \KAS$ is an isomorphism.
\end{cor}

\begin{proof}
Using Theorem \ref{bit}, we have a natural ring morphism $\gamma: \K \to \KN$. Then Theorem 4.5 in \cite{MCPI} induces the existence of $h: \Z[u] \to \KN$ such that $\gamma=h \circ \beta$. 
Note that the virtual Poincar\'e polynomial defines also a ring morphism $\beta^{\N}$ from $\KN$ to $\Z[u]$, and $\beta^{\N} \circ \gamma =\beta$. We are going to prove that $h$ is an inverse for $\beta^{\N}$, in a similar way as in the proof of Theorem \ref{thm-AS}. 
So let $X$ be a compact Nash manifold. By Nash-Tognoli Theorem, $X$ is Nash diffeomorphic to a compact nonsingular real algebraic variety $Y$, so that $[X]=[Y]$ in $\KN$. Then
$$h\circ \beta^{\N} ([X])=h \circ \beta^{\N} \circ \gamma ([Y])= h \circ \beta ([Y])=\gamma ([Y])=[X].$$
Similarly, if $P(u)$ is a polynomial in $\Z[u]$, then $P(u)=\beta(P(\LL))$. Then
$$\beta^{\N} \circ h (P(u))=\beta^{\N} \circ h \circ \beta (P(\LL))=\beta^{\N} \circ \gamma (P(\LL))$$
since $h\circ \beta=\gamma$. Moreover $\beta^{\N} \circ \gamma=\beta$ so that 
$$\beta^{\N} \circ h (P(u))=\beta(P(\LL))=P(u),$$
consequently the morphism $\KN \to \KAS$ is an isomorphism.
\end{proof}

A step further in the study of a Grothendieck ring is to be able to characterise an equality of classes. We know from \cite{vdd} that two semialgebraic sets are semialgebraically bijective if and only if they share the same Euler characteristic with compact supports and the same dimension. In the algebraic setting, a natural question, stated by Larsen and Lunts in \cite{LL}, asks whether the equality of the class of two algebraic varieties implies the existence of a piecewise isomorphism between those varieties. If the answer is negative in general \cite{Bo}, it happens to be true at the level of $\KAS$.

Let $A$ and $B$ be $\AS$-sets. We say that $A$ and $B$ are piecewise $\AS$-homeomorphic if there exist finite partitions $A=\sqcup_{i\in I} A_i$ and $B=\sqcup_{i\in I} B_i$ of $A$ and $B$ into $\AS$-sets such that for each $i\in I$, the sets $A_i$ and $B_i$ are $\AS$-homeomorphic.

\begin{thm}\label{LL} Let $A$ and $B$ be $\AS$-sets such that $[A]=[B]$ in $\KAS$. Then $A$ and $B$ are piecewise $\AS$-homeomorphic.
\end{thm}

\begin{proof} We prove the result by an induction on the dimension of $A$. 
Let $A$ and $B$ be $\AS$-sets such that $[A]=[B]$ in $\KAS$.
The virtual Poincar\'e polynomials of $A$ and $B$ are identical, so the dimension of $A$ and $B$ are equal using the $\AS$-analogue of Remark \ref{rem-surj}.\ref{beta-deg}. The result is then clearly true in dimension zero. Assume now that the dimension of $A$ is positive, and denote by $d$ that dimension. Compactify $A$ and $B$ in arc-symmetric sets $\overline A$ and $\overline B$ as in the proof of Lemma \ref{lem-class} (note that the dimension of $\overline A$ and $\overline B$ is still $d$). If the virtual Poincar\'e polynomial of $A$ and $\overline A$ are no longer equal, note however that their coefficient of degree $d$ is the same since the dimension of $\overline A\setminus A$ is strictly less than $d$ (and the same is true for $B$).

By the resolution of singularities, there exists a proper birational regular morphism $\pi:X\to \overline A^{Z}$ defined on a compact nonsingular real algebraic variety $X$ with value in the Zariski closure of $\overline A$ (view as a subset of some projective space), which is an isomorphism outside the exceptional divisor $E\subset X$ and a subvariety $F\subset \overline A^{Z}$ of dimension strictly smaller than $d$. Using Theorem 2.6 in \cite{K}, there exist a union of connected components $C$ of $X$ such that $\pi(C)$ is equal to the Euclidean closure of the regular points of $\overline A$. In particular, $C\setminus (C \cap E)$ and $\Reg \overline A$ are in bijection via the restriction of the regular isomorphism $\pi_{|X\setminus E}:X\setminus E\to \overline A^{Z}\setminus F$. Restricting $\pi$ further more, we see that there exist $\AS$-sets $C'$ and $A'$ of dimension strictly smaller than $d$ such that $C\setminus C'$ and $A\setminus A'$ are $\AS$-homeomorphic via $\phi=\pi_{|C\setminus C'}$.
In particular the coefficient of degree $d$ of $\beta(C)$ and $\beta(A)$ are equal.

Similarly, there exist a union of connected components $D$ of a compact nonsingular real algebraic variety of dimension $d$, and $\AS$-subsets $D'\subset D$ and $B'\subset B$ of dimension strictly smaller than $d$, such that $D\setminus D'$ and $B\setminus B'$ are $\AS$-homeomorphic via $\psi:D\setminus D' \to B\setminus B'$. Therefore $\beta(D)$ and $\beta(B)$ share the same coefficient of degree $d$.

As $C$ and $D$ are compact and nonsingular, their virtual Poincar\'e polynomial is equal to their (topological) Poincar\'e polynomial, so that the higher degree coefficient is equal to the constant coefficient by Poincar\'e duality. As a consequence, $C$ and $D$ have the same number of connected components, each of these components being a compact nonsingular connected Nash manifold of dimension $d$. Applying the strong factorisation theorem for compact connected Nash manifolds (Proposition 3.8 in \cite{MCP}), there exist two sequences of blowings-up $\sigma: \tilde C \to C$ and $\tau: \tilde D\to D$ along Nash centres such that $\tilde C$ and $\tilde D$ are Nash diffeomorphic. In particular there exist $\AS$-subsets $C''\subset C$ and $D''\subset D$ of dimension strictly smaller than $d$, such that $C\setminus C''$ and $D\setminus D''$ are $\AS$-homeomorphic, via the restriction $h:C\setminus C'' \to D\setminus D''$ of a Nash diffeomorphism.

We aim to deduce from the preceding results that there exists an $\AS$-homeomorphism between $A\setminus A_0$ and $B\setminus B_0$, for $\AS$-subsets $A_0\subset A$ and $B_0\subset B$ of dimension strictly smaller than $d$. If so, the virtual Poincar\'e polynomial of $A\setminus A_0$ and $B\setminus B_0$ would be equal, so that
$$\beta(A_0)=\beta(A)-\beta(A\setminus A_0)=\beta(B)-\beta(B\setminus B_0)=\beta(B_0),$$
and the proof could be achieved by the induction hypothesis applied to $A_0$ and $B_0$.

To prove that it is indeed the case, we begin with enlarging $C''$ in 
$$C'''=C''\cup C' \cup h^{-1}(D'\setminus D''),$$
and enlarging $D''$ in
$$D'''=D''\cup D' \cup h(C'\setminus C''),$$
so that $C'''$ and $D'''$ are still $\AS$-subsets of $C$ and $D$ of dimension strictly smaller than $d$ since $h$ is an $\AS$-bijection. Note that $C'''$ and $D'''$ satisfy that $C\setminus C'''$ and $D\setminus D'''$ are $\AS$-homeomorphic via (the restriction of) $h$, and moreover $C'$ is included in $C'''$ and $D'$ is included in $D'''$. Define finally $A_0$ and $B_0$ to be
$$A_0=A'\cup \phi (C'''\setminus C')$$
and 
$$B_0=B'\cup \psi (D'''\setminus D').$$
Then $A_0$ and $B_0$ are $\AS$-sets of dimension strictly smaller than $d$ (because $\phi$ and $\psi$ are bijective on $C'''\setminus C'$ and $D'''\setminus D'$ respectively),  they are included in $A$ and $B$ respectively, and are $\AS$-homeomorphic to $C\setminus C'''$ and $D\setminus D'''$ via the convenient restrictions of $\phi$ and $\psi$. Consequently $A\setminus A_0$ and $B\setminus B_0$ are $\AS$-homeomorphic as required.
\end{proof}

As a corollary, we obtain the following result (a positive answer to the $\AS$-version of a question by Gromov \cite{G}).

\begin{cor} Let $A$ and $B$ be $\AS$-sets included in a common $\AS$-set $C$. Assume that $C\setminus A$ and $C\setminus B$ are $\AS$-homeomorphic. Then $A$ and $B$ are piecewise $\AS$-homeomorphic.  
\end{cor}

In order to relate in section \ref{sect-cons} the relative Grothendieck rings of varieties with the corresponding rings of constructible functions, we define a relative analogue of $\KN$.

\begin{defi}
Let $S$ be a Nash set.
We define $\KNS$ to be the free abelian group generated by Nash isomorphism classes of proper Nash mapping $h:W\to S$, where $W$ is a Nash manifold, subject to the relations $[\emptyset \to S]=0$ and  
$$[h\circ \pi: \tilde W\to S]-[(h\circ \pi)_{|E}: E\to S]=[h:W\to S]-[h_{|C}:C \to S]$$
where $\pi: Bl_C W\to W$ denotes the blowing-up of $W$ along a Nash submanifold $C$ with exceptional divisor $E$.
The ring structure is induced by the fibred product over $S$. We denote by $\MNS$ the localisation of the class of the affine line in $\KNS$.
\end{defi}

\begin{rem}\label{rem-blN}
\begin{enumerate}
\item We have a natural ring morphism $\KS\to \KNS$ induced by the presentation of $\KS$ in terms of proper mappings defined on nonsingular varieties (Theorem \ref{bit}). 
\item In the case the variety $S$ is reduced to a point, the ring $\KNS$ is nothing else than $\KN$, which is isomorphic to $\KAS$ by Corollary \ref{cor-AS}. 
\item \label{QblN} We have also a natural morphism $\KNS\to \KASS$ in the relative case. We do not know however whether this morphism is an isomorphism. The additional difficulty raises in the mixed situation between Nash mappings and $\AS$-mappings.
\end{enumerate}
\end{rem}

We conclude this section by collecting the relationships obtained so far between the different Grothendieck rings. 

\begin{prop} Let $S$ be a real algebraic variety. The following diagram is commutative
$$\xymatrix{\KblS \ar[d]_{\simeq}\ar[r] & \KNS \ar[d] & \\
\KS \ar[d]\ar[r] & \KASS \ar[d]\ar[r] & \KSAS \ar[d]\\
\K \ar[r]^{\alpha}\ar[d]_{\beta} & \KAS \ar[r]\ar[d]_{\bAS}& \KSA \ar[d]_{\chi_c}\\
\Z[u] \ar[r]^{\id} & \Z[u] \ar[r]^{u\mapsto -1} & \Z }$$
\end{prop}


\section{Motivic Milnor fibres with sign}\label{sectZ}

Local motivic zeta functions in real geometry have been under interest in the past years \cite{KoP,F,Pr,C}, notably because they give rise to powerful invariants in the study of real singularity theory (cf. \cite{P2} and \cite{C2} for two recent classification results, for simple singularities in the equivariant case, and for non-degenerate weighted homogeneous polynomial with respect to the arc-analytic equivalence respectively). A contrario, the study of motivic Milnor fibres is less developed at the moment, and only its relation with the topological Milnor fibres is understood when the measure is the Euler characteristic with compact supports \cite{CF}.
We introduce in this paper the global motivic zeta functions with sign, define the global motivic Milnor fibres with sign and make the connection with the local ones. We provide also a formula for the motivic Milnor fibres with sign associated with a convenient weighted homogeneous polynomial non-degenerate with respect to its Newton polyhedron.
These results will be use in section \ref{sect-dual}.

\subsection{Zeta functions with signs}

Let $f:X\to \R$ be a polynomial function defined on a nonsingular real algebraic variety $X$. Let $X_0=f^{-1}(0)$ be the zero set of $f$. Denote by $\mathcal L_n(X)$ the truncated arc space of $X$ at order $n\in \mathbb N$ (cf. \cite{DL} or \cite{F} in the real case).

The zeta functions with sign $Z_f^{+}$ and $Z_f^{-}$ of $f$ are defined by
$$Z_f^{\pm}(T)= \sum _{n \geq 1} \LL^{-nd} [\mathcal X_n^{\pm}\to X_0]T^n \in \M_{X_0}[[T]],$$
where 
$$\mathcal X_n^{\pm} =\{\gamma \in  \mathcal L_n(X): f\circ \gamma
(t)=\pm t^n+\cdots \}$$
and $\mathcal X_n^{\pm} \to X_0$ is the map defined by $\gamma \mapsto \gamma(0)\in X_0$.

\begin{rem} Given a point $x_0\in X_0$, the pull-back morphism $\M_{X_0} \to \M$ induced by the inclusion $\{x_0\} \to X_0$ sends $Z^{\pm}_f$ to the local motivic zeta function with sign $Z_{f,x_0}^{\pm}$ considered in \cite{F}. 
\end{rem}

The zeta functions with sign can be computed on a resolution of singularities. Let $\sigma:M \to X$ be an embedded resolution of $X_0$, such that $f \circ \sigma$ and the
Jacobian determinant $\jac \sigma$ are normal crossings simultaneously,
and assume moreover that $\sigma$ is an isomorphism over the
complement of the zero locus of $f$.

Let $(f \circ \sigma)^{-1}(0)= \cup_{k \in K}E_k$ be the decomposition
into irreducible components of $(f \circ \sigma)^{-1}(0)$.
Put $N_k=\mult _{E_k}f \circ \sigma$ and $\nu _k=1+\mult _{E_k} \jac
\sigma$, and for $I \subset K$ denote by $E_I^0$ the set $(\cap _{i
\in I} E_i) \setminus (\cup _{k \in K \setminus I}E_k)$.

When we are dealing with signs, one defines coverings  $\widetilde
{E_I^{0,\pm}}$ of $E_I^0$, in
the following way (cf. \cite{F}). 
Let $U$ be an affine open subset of $M$
such that $f \circ \sigma=u x_1^{m_1}\cdots x_k^{m_k}$, with $k>0$ and $m_i>0$, where $x_1,\ldots,x_n$ are local coordinates and $u$ is a unit, and suppose that $I$ is given by $I=\{1,\ldots,l\}$, with $l\leq k$. Let us put $$R_{U}^{\pm}=\{ (x,t) \in (E_I^0 \cap U) \times \mathbb
R; t^{m_I} u_I(x)  =\pm 1\},$$ where $m_I=gcd(N_i)$ and $u_I(x)=u \prod_{j\notin I}x_j^{m_j}$. Then the $R_{U}^{\pm}$ glue
together along the $E_I^0 \cap U$ to give $\widetilde {E_I^{0,\pm}}$. We can now state the real version of Denef-Loeser formula.

\begin{thm}\label{thmDL} With the notation introduced upstairs, the motivic zeta functions with sign satisfy
$$Z_f^{\pm}(T)=\sum_{I\neq \emptyset} (\LL-1)^{|I|-1}[\widetilde{E_I^{0,\pm}}\to X_0] \prod_{i \in I}\frac{\LL^{-\nu_i}T^{N_i}}{1-\LL^{-\nu_i}T^{N_i}}$$
in $\M_{X_0}[[T]]$.
\end{thm}

\begin{rem} Consider the $\M_{X_0}$-submodule of $\M_{X_0}[[T]]$ generated by $1$ and finite products of terms $\frac{\LL^{a}T^{k}}{1-\LL^{a}T^{l}}$ with $a\in \Z$ and $k\leq l \in \mathbb N$. Denote by $\mathcal E$ the set $\mathcal E=\{(a,k,l)\in \Z \times \mathbb N^2:~k\leq l\}$. There exists a unique $\M_{X_0}$-linear morphism
$$\M_{X_0}\big[ \frac{\LL^{a}T^{k}}{1-\LL^{a}T^{l}}  \big]_{(a,k,l)\in \mathcal E} \to \M_{X_0}$$
mapping
$$\prod_{(a_i,k_i,l_i)\in I} \frac{\LL^{a_i}T^{k_i}}{1-\LL^{a_i}T^{l_i}}$$
to $(-1)^{|I|}$ if $k_i=l_i$ for any $i$, and to $0$ otherwise, for each finite set $I\subset \mathcal E$. The image of an element is call its limit as $T$ tends to $\infty$. Its corresponds to the constant coefficient in its Taylor development in $1/T$ (cf. \cite{NS}, Definition 8.1).
\end{rem}

The limit of $Z_f^{\pm}(T)$ as $T$ goes to infinity makes sense in $\M_{X_0}$, and we denote it by $-\psi_f^{\pm}\in \M_{X_0}$. Then $\psi_f^{+}$, respectively $\psi_f^{-}\in \M_{X_0}$, is called the positive, respectively negative, motivic Milnor fibre of $f$. As a consequence of Theorem \ref{thmDL}, we have the following expression for the motivic Milnor fibres with sign.

\begin{cor}\label{cor-fib} The expression
$$\psi_f^{\pm}=\sum_{I\neq \emptyset} (\LL-1)^{|I|-1}[\widetilde{E_I^{0,\pm}}\to X_0]\in \M_{X_0}$$
does not depend on the resolution of the singularities $\sigma$ of $X_0$.
\end{cor}

\begin{ex}\label{ex-mot}\begin{flushleft}\end{flushleft}
\begin{enumerate} 
\item Let $f:\mathbb R^2 \longrightarrow
    \mathbb R$ be defined by $f(x,y)=x^2+y^4$. Then $X_0$ is reduced to the origin, and one resolves the singularities of $X_0$ by two successive pointwise blowings-up. The exceptional divisor $E$ has 
    two irreducible components $E_1$ and $E_2$ with $N_1=2,~ \nu_1=2,
   ~ N_2=4, ~\nu _2=3$. Then (correcting a mistake in \cite{F}) $\widetilde E_1^{0,+}$ is a double covering of $\PP^1$ (minus one point) given by a union of two $\PP^1$ (each one minus one point), whereas $\widetilde E_2^{0,+}$ is isomorphic to the double covering of $\PP^1$ (minus one point) given by the boundary of a Moebius band, so that
$$Z_f^+(T)=2(\LL-1)\frac{\LL^{-2}T^2}{1-\LL^{-2}T^2}\frac{\LL^{-3}T^4}{1-\LL^{-3}T^4}+2\LL
\frac{\LL^{-2}T^2}{1-\LL^{-2}T^2}+(\LL-1) \frac{\LL^{-3}T^4}{1-\LL^{-3}T^4}.$$
As a consequence $\psi_f^+=\LL+1\in \M$.
\item \label{ex2} Let $f:\R^2\to \R$ be defined by $f(x,y)=x^6+x^2y^2+y^6$. The zero set of $f$ is again reduced to the origin. The positive zeta function of $f$ can be computed using Denef \& Loeser formula:
$$Z_f^+(T)=2(\LL-1)\frac{\LL^{-2}T^4}{1-\LL^{-2}T^4}+2(\LL-1)\frac{\LL^{-3}T^6}{1-\LL^{-3}T^6}+4(\LL-1)\frac{\LL^{-2}T^4}{1-\LL^{-2}T^4}\frac{\LL^{-3}T^6}{1-\LL^{-3}T^6},$$
so that $\psi_f^+=0\in \M$.
\item Let $f:\mathbb R^2 \longrightarrow
    \mathbb R$ be defined by $f(x,y)=y^2+x^2(x^2-1)$. Then $X_0$ is a compact curve homeomorphic to a figure "eight", singular at the origin. We resolve the singularity of $X_0$ by blowing-up the origin in $\R^2$, giving rise to the strict transform $E_1$ which intersects the exceptional divisor $E_2\simeq \PP^1$ in two points. Then $\tilde E_{1,2}^{0,+}$ and $\tilde E_{1}^{0,+}$ are isomorphic to $E_{1,2}$ and $E_1^0$ respectively, whereas $\tilde E_{2}^{0,+}$ is isomorphic to a circle minus two points covering the open interval in $E_2$ around which the pull-back of $f$ is positive. Then
$$\psi_f^+=[E_{1}^{0}\to X_0]+[\tilde E_{2}^{0,+}\to X_0]-(\LL-1)[ E_{1,2}\to X_0]\in \M_{X_0}.$$
We can smoothly compactify $E_{1}^{0}\to X_0$ and $\tilde E_{2}^{0,+}\to X_0$ in $E_{1}\to X_0$ and $E_2\to X_0$ adding the two missing points, therefore another expression for $\psi_f^+$ is given by
$$\psi_f^+=[E_{1}\to X_0]+[E_{2}\to X_0]-(\LL+1)[ E_{1,2}\to X_0].$$
\end{enumerate}
\end{ex}

\subsection{Case of a weighted homogeneous polynomial}

In this section, we state a formula for the local motivic Milnor fibres associated with a polynomial function non-degenerate with respect to its Newton polyhedron (we will use such results in section \ref{sect-dual}), using results in \cite{FF}. We begin with some notation.
Let $f:\R^d\to\R$ denote a polynomial function vanishing at the origin.
Consider its Taylor expansion at the origin of $\R^d$,
$$
f(x)=\sum_{\nu \in (\mathbb N\cup \{0\})^d}c_\nu x^\nu,
$$
where  $c_\nu \in \R$. For a subset $S$ of $(\mathbb N\cup \{0\})^d$ we set
$$
f_S(x)=\sum_{\nu\in S}c_\nu x^\nu.
$$
Let $\Gamma_f$ denote the Newton polyhedron of $f$, namely the convex hull of the set
$$
\cup_{\nu \in (\mathbb N\cup \{0\})^d}\{(\nu+\R_+^d):c_\nu\ne0\}. 
$$
The Newton boundary $\Gamma_f^c$ of $f$ is the union of the compact faces of $\Gamma_f$. 
For $a\in \mathbb R^d_+$ and 
$\nu\in \mathbb R^d$, we denote by
$\langle a,\nu\rangle$ the usual scalar product on $\R^d$ and define the
multiplicity $m_f(a)$ of $f$ relative to $a$ by
$$
m_f(a)=\min\{\langle a,\nu\rangle :\nu\in\Gamma_f\}.
$$
The face $\gamma_f(a)$ 
of the Newton polyhedron of $f$ associated with $a\in \mathbb R_+^d$ is defined by
$$
\gamma_f(a)=\{\nu\in\Gamma_f:\langle a,\nu\rangle=m_f(a)\}.$$
The cone $\sigma=\{a\in\R_+^d:\gamma(a)\supset\gamma\}$ associated with the face $\gamma$ of the Newton polyhedron of $f$ is called the dual cone of $\gamma$. The dual of $\Gamma_f$ is the union of the dual cones associated to the faces of $\Gamma_f$.

We say that the polynomial function $f$ is non-degenerate with respect to its Newton polyhedron
if, for any $\gamma\in \Gamma_f^c$, 
all singular points of $f_\gamma$ are contained in the union of 
some coordinate hyperplanes. 
Namely $f$ is non-degenerate if
$$
(
\frac{\partial f_\gamma}{\partial x_1}(c),\dots,\frac{\partial f_\gamma}{\partial x_d}(c)
)\neq (0,\dots,0)
$$
for any $\gamma\in \Gamma_f^c$ and any $c\in (\R^*)^d$ with $f_\gamma(c)=0$.

Finally, for any $\gamma\in \Gamma_f^c$ , we define some algebraic subsets $X_\gamma$, $X^+_\gamma$ and $X^-_\gamma$ of $(\R^*)^d$ 
by
$$
X_\gamma=\{c\in(\R^*)^d:f_\gamma(c)=0\}$$
and
$$
X^{\pm}_\gamma=\{c\in(\R^*)^d:f_\gamma(c)=\pm1\}.
$$

\begin{rem}
In the case that $\gamma$ is included in exactly $p$ coordinate hyperplanes, note that $X_\gamma$ and $X^{\pm}_\gamma$ are a product of $(\R^*)^p$ times the algebraic subsets $\hat X_\gamma$ and $\hat X^{\pm}_\gamma$ of $(\R^*)^{d-p}$ defined with the same equations as $X_\gamma$ and $X^{\pm}_\gamma$, but considering only the remaining variables.
\end{rem}

Then we have the following expression for the local zeta functions with signs of a polynomial function $f$ whose associated dual Newton polyhedron is generated by simplicial cones, using the same strategy as Bories and Veys in \cite{BV} in order to keep an expression in $\M[[T]]$ (note that the formula in \cite{BV} is even more general, they do not require the simplicial assumption). 

For $a\in \Z^d$, we set $s(a)=\sum_{i=1}^da_i$. If a cone $\sigma$ is of the form $\sigma=\R_+ v_1+\cdots+\R_+ v_l$, denote by $Q_\sigma$ the set
$$Q_\sigma=
\{\lambda_1{v}_1+\dots+\lambda_l{v}_l: 0<\lambda_i \leq 1, i=1,\dots,l\}.$$

\begin{prop}\label{lemP} Let $f:\R^d\to \R$ be a polynomial function non-degenerate with respect to its Newton polyhedron $\Gamma_f$. Assume that the dual cones of the face $\gamma\in \Gamma_f^c$ are simplicial. Then the local zeta functions with sign of $f$ satisfy
$$
Z_{f,0}^{\pm}
=\sum_{\gamma \in \Gamma^c_f} ([\hat X_\gamma^{\pm}]+\frac{\LL^{-1}T}{1-\LL^{-1}T}[\hat X_\gamma])  P(\gamma)\in \M[[T]]
$$
where 
$$
P(\gamma)=
\frac{\sum_{a\in Q_\sigma\cap \Z^n} \LL^{-s(a)} T^{m_f(a)} }
{\prod_{i=1}^k (1-\LL^{-s({v}_i)} T^{m_f({v}_i)})}\in \M[[T]]
$$
if the dual cone $\sigma$ of $\gamma$ is $(p+k)$-dimensional of the form 
$$\sigma=\R_+{e}_{i_1}+\dots+\R_+{e}_{i_p}+\R_+{v}_1+\cdots+\R_+{v}_k$$ 
with $m_f({e}_{i_j})=0$ exactly for $j\in \{1,\dots,p\}$.
\end{prop}

We use it below to state a simple expression for the motivic Milnor fibres with sign associated with a convenient weighted homogeneous polynomial non-degenerate with respect to its Newton polyhedron. A polynomial $f$ is convenient if $\Gamma_f$ intersects all coordinates axis.

\begin{prop}\label{propM} Let $f\in \R[x_1,\ldots,x_d]$ be a convenient weighted homogeneous polynomial non-degenerate with respect to its Newton polyhedron. Then
$$\psi_{f,0}^{\pm}=[\{f=\pm 1\}]-[\{f=0\}]+1\in \M.$$
\end{prop}

\begin{proof}
Denote by $(w_1,\ldots,w_d)$ the weights of $f$, chosen in such a way that the weight vector $w=(w_1,\ldots,w_d)$ is primitive. For a strict subset $I\subsetneq \{1,\ldots,d\}$, denote by $\sigma_I$ the cone generated by $w$ and the $e_i$'s for $i$ belonging to $I$.  
The corresponding face $\gamma_I$ is the intersection of $\Gamma_f$ with the coordinate hyperplanes $x_i=0$, with $i\in I$.

Under our assumptions, the dual of the Newton polyhedron of $f$ is the union of the cones $\sigma_I$, with $I\subsetneq \{1,\ldots,d\}$. Passing to the limit as $T$ goes to infinity in the expression given by Proposition \ref{lemP}, we see that the terms $P(\gamma_I)$ tend to $-1$ 
(here $k=1$ and only $a=w\in Q_{\sigma_I}$ contributes to the limit).
Therefore the motivic Milnor fibre with sign of $f$ is given by
$$\psi_{f,0}^{\pm}=\sum_{I\subsetneq \{1,\ldots,d\}} ([\hat X_{\gamma_I}^{\pm}]-[\hat X_{\gamma_I}]).$$
This formula implies the result by additivity in $\M$ since 
$$\{f=\pm 1\}=\sqcup_{I\subsetneq \{1,\ldots,d\}} \hat X_{\gamma_I}^{\pm}$$
whereas 
$$\{f=0\}=\{0\} \sqcup_{I\subsetneq \{1,\ldots,d\}} \hat X_{\gamma_I}.$$
\end{proof}

\begin{rem} We can use Proposition \ref{lemP} to compute the motivic Milnor fibre for non necessarily weighted homogeneous polynomial. For example, consider the function $f:\R^2\to \R$ defined by $f(x,y)=x^6+x^2y^2+y^6$. Set $v_1=(2,1)$ and $v_2=(1,2)$. The dual of $\Gamma_f$ is the union of five cones generated respectively by $e_1$ and $v_1$, $v_1$, $v_1$ and $v_2$, $v_2$, $v_2$ and $e_2$. Note that $X_{\gamma}$ is empty for any compact face $\gamma$ of $\Gamma_f$, so that passing to the limit in the formula given by Proposition \ref{lemP} implies the following expression for $\psi_{f,0}^+$ :
$$\psi_{f,0}^+=\frac{-1}{\LL-1}[y^6=1]-[x^2y^2+y^6=1]+(-1)^2[x^2y^2=1]-[x^6+x^2y^2=1]+\frac{-1}{\LL-1}[x^6=1],$$
where all sets are considered in $(\R^*)^2$. Therefore
$$\psi_f^+=2(\LL-3)-2[(x,y)\in (\R^*)^2:x^6+x^2y^2=1]\in \M.$$
Note that we can compute the virtual Poincar\'e polynomial of the remaining term (the compactification of the plane curve $x^6+x^2y^2=1$ in the projective space gives a curve with one connected component with a double point at infinity), so that 
$$\beta(\psi_f^+)=2(u-3)-2(u-3)=0\in \KAS.$$
\end{rem}
\section{Constructible functions}\label{sect-cons}

We begin this section by recalling the definition of (semialgebraically) constructible functions as developed by Schapira \cite{Sch} (in the subanalytic case) and Viro \cite{V}, algebraically constructible functions and Nash constructible functions as developed by McCrory and Parusi\'nski \cite{MCP}. The latter have been proven to be interesting invariants in the study of the topology of real algebraic sets. We focus in this section of the relationships between constructible functions and the relative Grothendieck rings introduced previously.

\subsection{Semialgebraically constructible functions}
Let $S\subset \R^n$ be a semialgebraic subset of $\R^n$.
A semialgebraically constructible function on $S$, or simply a constructible function on $S$, is an integer-valued function $\phi:S\to \Z$ which takes finitely many values and such that, for each $n\in \Z$, the set $\phi^{-1}(n)$ is a semialgebraic subset of $S$. A constructible function $\phi$ on $S$
can be written as a finite sum $\phi=\sum m_i\un_{S_i}$ where, for each $i$, the set $S_i$ is a semialgebraic subset of $S$, the function $\un_{S_i}$ is the characteristic function of $S_i$ and $m_i$ is an integer. The set $F(S)$ of constructible functions on $S$ is a ring under pointwise sum and product. The pull-back $f^*\phi$ of a constructible function $\phi\in F(T)$ under a continuous semialgebraic mapping $f: S \to T$ is the constructible function $f^*\phi=\phi \circ f$ obtained by composition with $f$. It induces a ring morphism $f^*:F(T)\to F(S)$. 
The push-forward is defined using integration along the Euler characteristic, which we define first. The Euler integral of a constructible function $\phi \in F(S)$ over a semialgebraic subset $A\subset S$ is defined by the finite sum
$$\int_A \phi=\sum_{n\in \Z} n \chi_c(\phi^{-1}(n)\cap A)\in \Z.$$
Given a continuous semialgebraic mapping $f: S \to T$, the push-forward $f_!\phi$ of $\phi\in F(S)$ is the constructible function from $T$ to $S$ defined by
$$f_!\phi (t)=\int_{f^{-1}(t)}\phi.$$ 
It induces a group morphism $f_!:F(S) \to F(T)$.

We recall the definition of two important operations defined on the ring of constructible functions: the duality and link operators. 
The duality is a group morphism $D_S:F(S) \to F(S)$ defined as follows. For $\phi\in F(S)$, the dual $D_S\phi$ of $\phi$ is the function defined by
$$s\mapsto \int_{B(s,\epsilon)}\phi ,$$
where $B(s,\eps)$ is the open ball centred at $s\in S$ with radius $\epsilon>0$, and $\epsilon$ is chosen small enough. This integral is well-defined thanks to the local conical structure of semialgebraic sets, so that $D_S\phi$ is well-defined and it is indeed a constructible function. Moreover the duality is involutive on $F(S)$.

The link is a group morphism $\Lambda_S:F(S)\to F(S)$ defined similarly. For $\phi\in F(S)$, the link $\Lambda_S\phi$ is the constructible function defined by
$$s\mapsto \int_{\Sph(s,\epsilon)}\phi ,$$
where $\Sph(s,\eps)$ is the open ball centred at $s\in S$ with radius $\epsilon>0$,
with $\epsilon$ small enough. Note that the duality and link operators satisfy the relation $D_S+\Lambda_S=\id$ on $F(S)$.

\begin{rem} For a semialgebraic subset $A$ of $S$, the link $\lk(s,A)$ of $A$ at $s\in S$ is the semialgebraic set $A\cap\Sph(s,\epsilon)$, where $\epsilon$ is chosen sufficiently small so that $\lk(s,A)$ does not depend on $\epsilon$, by the local conical structure of semialgebraic sets. Note actually that $\lk(s,A)$ does not depend either on the distance function to define the sphere.
\end{rem}

\begin{ex}\label{ex-simplex}\begin{enumerate}
\item For a point $a\in S$, we have $D_S \un_a=\un_a$ and $\Lambda_S \un_a=0$.
\item \cite{Co} Consider the characteristic function of a $d$-dimensional open simplex $\sigma$, with closure $\overline \sigma$ included in $S$. Then 
$$D_S \un_{\sigma} =(-1)^d \un_{\overline \sigma} \textrm{~~~and~~~} \Lambda_S \un_{\sigma}=(-1)^{d-1} \un_{\overline \sigma} +\un_{\sigma},$$
whereas
$$D_S \un_{\overline \sigma} =(-1)^d \un_{\sigma} \textrm{~~~and~~~} \Lambda_S \un_{\overline \sigma}= \un_{\overline \sigma} +(-1)^{d-1}\un_{\sigma}.$$
\item If $X$ is a $d$-dimensional nonsingular real algebraic subset of $S$, then $D_S \un_X=(-1)^d\un_X$ and $\Lambda_S \un_X=(1+(-1)^{d-1})\un_X$.
\end{enumerate}
\end{ex}

We recall some properties of $D_S$ and $\Lambda_S$ whose proof can be found in \cite{MCP}.

\begin{prop}\label{prop-const}(\cite{MCP})\begin{enumerate}
\item Let $f: S \to T$ be a continuous proper semialgebraic mapping. Then $D_S$ and $\Lambda_S$ commute with $f_!$.
\item $D_S\circ D_S=\id$ and $\Lambda_S \circ \Lambda_S=2\Lambda_S$.
\item $\Lambda_S \circ D_S=-D_S \circ \Lambda_S$.
\item If $S$ is compact, then $\int_S \Lambda_S \phi=0$ for any $\phi \in F(S)$.
\end{enumerate} 
\end{prop}

\begin{rem}\label{rem-point} If $S$ is a point, then the duality corresponds to the identity map on $F(S)\simeq \Z$.
\end{rem}

By the very definition, we have $\Lambda_S(\un_A)(s)=\chi_c(\lk(s,A))$.
As a consequence of Fubini Theorem for constructible function (Theorem 3.5 in \cite{Co}), we have therefore the following result.

\begin{prop}\label{prop-fub} For $h:X\to S$ is a semialgebraic mapping and $s\in S$, we have 
$$\Lambda_S h_! \un_X (s)=
\chi_c(h^{-1}(\lk(s,S))).$$
\end{prop}

\subsection{Algebraically and Nash constructible functions}

Algebraically constructible functions have been defined in \cite{MCP} as a subclass of the class of constructible functions on $S$, stable under push-forward along regular mappings. More precisely, the algebraically constructible functions are those constructible function $\phi$ on $S$ which can be described as a finite sum 
$$\phi=\sum m_i (f_i)_! \un_{X_i},$$
where $f_i:X_i\to S$ are regular mappings from real algebraic varieties $X_i$ (note that, as in \cite{Co}, we do not require $f_i$ to be proper) and $m_i$ are integers. Algebraically constructible functions on $S$ form a subring $A(S)$ of $F(S)$. A crucial result in \cite{MCP} is that the link operator maps $A(S)$ to itself.

Nash constructible functions are defined similarly to algebraically constructible functions, by allowing to restrict the regular functions $f_i:X_i\to S$ to the larger class of connected components $W_i\subset X_i$ of the algebraic variety $X_i$. Nash constructible functions form a subring $N(S)$ of $F(S)$ containing $A(S)$, stable by push-forward and the link operator. 

The properties of algebraically constructible functions and Nash constructible functions are closely related to the properties of Zariski constructible sets and arc-symmetric sets. In particular, if $A\subset S$ is semialgebraic, then $\un_A$ is algebraically constructible if and only if $A$ is Zariski constructible, whereas $\un_A$ is Nash constructible if and only if $A$ belong to $\AS$ \cite{P}.

\subsection{Relation with the relative Grothendieck rings}
As noticed by Cluckers and Loeser in the introduction of \cite{CL}, Proposition 1.2.2, the Grothendieck ring of semialgebraic sets over a given semialgebraic set $S$ is isomorphic to the ring of constructible functions on $S$.

\begin{prop}\label{K0-const}(\cite{CL}) Let $S$ be semialgebraic set. The mapping $[h:A\to S] \mapsto h_!(\un_A)$ induces a ring isomorphism $\Pi: \KSAS\to F(S)$. Under that isomorphism, the push-forward $f_!:F(S)\to F(T)$ associated with a semialgebraic mapping $f:S\to T$ corresponds to the morphism $\KSAS\to \KSAT$ induced by composition with $f$.
\end{prop}

\begin{rem} \begin{enumerate}
\item Under the isomorphism $\Pi$, the Euler integral on $F(S)$ corresponds to the push-forward on a point. Actually, $\int:F(S)\to \Z$ is equal to $p_!:F(S)\to F(\{p\})$ where $p:S\to \{a\}$ is the projection onto a point $a$.
\item The Euler integral over a semialgebraic set $A\subset S$ corresponds to the composition $p_! \circ i_A^*$ where $p:A\to \{a\}$ is the projection onto a point $a$ and $i_A: A \hookrightarrow S$ denotes the inclusion of $A$ in $S$.
\end{enumerate}
\end{rem}

We can express the duality and link operators on $\KSAS$ in terms of the duality and link on constructible functions, using their commutativity with proper push-forward. To this aim, we need to express a class in $\KSAS$ in terms of classes of projection mappings.

\begin{lem}\label{lem-tri} Let $h:X\to S$ be a continuous semialgebraic mapping. There exist a semialgebraic triangulation $S=\sqcup S_i$ of $S$, and semialgebraic sets $F_i$, such that
$$[h:X\to S]=\sum [p_i:F_i\times S_i \to S],$$
where $p_i$ denotes the projection onto the second coordinate.
\end{lem}

\begin{proof} 
By semialgebraic Hardt triviality \cite{BCR}, there exist a partition $S=\sqcup S_i$ of $S$ such that $h$ is trivial over each $S_i$, meaning that there exist semialgebraic sets $F_i$ and homeomorphism $f_i:F_i\times S_i \to h^{-1}(S_i)$ such that $h \circ f_i$ corresponds to the projection $p_i:F_i\times S_i$ onto the second coordinate. We refine the partition of $S$ into a semialgebraic triangulation, such that each $S_i$ is either a point or an open simplex. Then
$$[h:X\to S]=\sum [h_{|h^{-1}(S_i)}:h^{-1}(S_i)\to S]$$
by additivity, whereas
$$[h_{|h^{-1}(S_i)}:h^{-1}(S_i)\to S]=[p_i:F_i\times S_i \to S],$$
by triviality, so that the conclusion follows.
\end{proof}

Denote by 
$$\DSAS:\KSAS\to \KSAS$$ 
the duality induced by $D_S:F(S)\to F(S)$ on $\KSAS$ via $\Pi$, namely $\DSAS=\Pi^{-1}\circ  D_S \circ \Pi$. Define similarly $\LSAS$ on $\KSAS$. We are going to describe the action of $\DSAS$ and $\LSAS$  on $\KSAS$.

\begin{rem}\label{rem-alg}\begin{enumerate}
\item If $X$ is a nonsingular algebraic set and $h:X\to S$ is a proper semialgebraic mapping, then 
$$\DSAS [h:X\to S]=(-1)^{\dim X} [h:X\to S].$$ 
Actually $D_S$ and $h_!$ commute because $h$ is proper, and $D_S \un_X=(-1)^{\dim X} \un_X$ by Example \ref{ex-simplex}, so that
$$\DSAS [h:X\to S]=\Pi^{-1} ((-1)^{\dim X}h_!(\un_X))=(-1)^{\dim X}[h:X\to S].$$
\item Similarly
$$\LSAS [h:X\to S]=(1-(-1)^{\dim X}) [h:X\to S]=\chi_c(\Sph^{\dim X-1})[h:X\to S].$$ 
\end{enumerate}
\end{rem}

More generally, using Lemma \ref{lem-tri}, we can describe the action of $\DSAS$ as follows.

\begin{prop} Let $\sigma \subset S$ be a $d$-dimensional open simplex in a semialgebraic triangulation of $S$ and let $p:F\times \sigma \to S$ denote the projection onto the second coordinate. Then
$$\DSAS [p:F\times \sigma \to S]=(-1)^d [p:F\times (\overline \sigma \cap S) \to S]$$
and
$$\LSAS [p:F\times \sigma \to S]=(-1)^{d-1} [p:F\times (\overline \sigma \cap S) \to S]+[p:F\times \sigma \to S],$$
where $\overline \sigma$ denotes the Euclidean closure of $\sigma$.
\end{prop}

\begin{proof}
The image of $[p:F\times \sigma \to S]$ by $\Pi$ is equal to $\chi_c(F) \un_{\sigma}$ by a direct computation. The dual of $\un_{\sigma}$ has been computed in Example \ref{ex-simplex} in the particular case where the closure of $\sigma$ is included in $S$. In the general case, one need to restrict ourself to the boundary of $\sigma$ which in included in $S$, so that $D_S \un_{\sigma}=(-1)^d \un_{\overline \sigma \cap S}$. Finally $\chi_c(F) \un_{\overline \sigma \cap S}$ is the image under $\Pi$ of the element $[p:F\times (\overline \sigma \cap S) \to S]$ in $\KSAS$.
\end{proof}

We can state the analogue of Proposition \ref{K0-const} for algebraically constructible and Nash constructible functions, establishing the link between the Grothendieck rings and the various rings of constructible functions.

\begin{thm}\label{prop-alg} Let $S$ be a real algebraic variety. The assignment $[h:X\to S] \mapsto h_!(\un_X)$ induces surjective ring morphisms from $\KS$ to $A(S)$, from $\KNS$ to $N(S)$ and from $\KASS$ to $N(S)$. In particular, we have a commutative diagram
$$\xymatrix{\KS \ar@{->>}[d]\ar[r] & \KNS \ar@{->>}[d]\ar[r] & \KASS \ar[r]\ar@{->>}[d] & \KSAS \ar[d]^{\simeq}\\
A(S) \ar@{^{(}->}[r] & N(S) \ar[r]^{=} & N(S)\ar@{^{(}->}[r] & F(S) \\}$$
\end{thm}

\begin{proof}
Let $X$ be a set, either algebraic, Nash, $\AS$ or semialgebraic, and $h:X\to S$ be a mapping, either regular, Nash, $\AS$ or semialgebraic respectively.
The assignment $[h:X\to S] \mapsto h_!(\un_X)\in F(S)$ defines a group morphism because for $Y$ closed in $X$, $Y$ being either a algebraic, Nash, $\AS$ or semialgebraic subset, we have $\un_X=\un_Y +\un_{X\setminus Y}$ and $h_!:F(X)\to F(S)$ is a group morphism. 

Concerning the product, let $h_i:X_i\to S$, with $i\in \{1,2\}$, be a mapping, both either regular, Nash, $\AS$ or semialgebraic. Denote by $X=X_1 \times_S X_2 $ the product of $X_1$ and $X_2$ over $S$, so that $[h:X\to S]$ is the product of the classes of $h_1:X_1\to S$ and $h_2:X_2 \to S$. For $s\in S$, we have
$$h_!(\un_X)(s)=\chi_c(h^{-1}(s))$$
and $h^{-1}(s)=h_1^{-1}(s) \times h_2^{-1}(s)$ so that
$$h_!(\un_X)(s)=(h_1)_!(\un_X)(s)\cdot (h_2)_!(\un_X)(s)$$
by additivity of $\chi_c$ and the mapping considered is indeed a ring morphism. 

The fact that the target of the morphism defined on $\KS$ by the assignment $[h:X\to S] \mapsto h_!(\un_X)$ is $A(S)$ follows directly from the definitions.
In the case of $\KASS$ and $\KNS$ however, we need to prove that $h_!(\un_X)$ is a Nash constructible function on $S$ whenever $h$ is an $\AS$-mapping from an $\AS$-set $X$, or a proper Nash mapping from a Nash manifold. It is sufficient to prove this fact in the former case, and to show this, we adapt the proof of \cite{KP}, Corollary 2.13.
Consider the graph $\Gamma_h$ of $h$ as a subset of $X\times S\subset \overline X^{Z}\times S$, and denote by $\pi_S :\overline X^{Z}\times S\to S$ the projection onto $S$. Then, notice that $h_!(\un_X)={\pi_S}_! (\un_{\Gamma_h})$, and that $\un_{\Gamma_h}$ is a Nash constructible function on $\overline X^{Z}\times S$ because $\Gamma_h$ is an $\AS$-set (cf. \cite{KP}, Theorem 2.9). As a consequence, $h_!(\un_X)$ is a Nash constructible function on $S$ as the push-forward of a Nash constructible function along a regular mapping.

The surjectivity of $\KS \to A(S)$, $\KNS \to N(S)$ and $\KASS \to N(S)$ are clear.
\end{proof}

The surjective morphisms on $A(S)$ and $N(S)$ considered in Theorem \ref{prop-alg} are not injective, as illustrated by the following example.

\begin{ex} For $S=\Sph^1$, consider the $2$-covering $h_1:\Sph^1 \cup \Sph^1 \to \Sph^1$ induced by the identity on each $\Sph^1$ and the $2$-covering $h_2:\Sph^1 \to \Sph^1$ induced by the boundary of a Moebius band. The image of $h_1$ and $h_2$ in $A(\Sph^1)$ are both equal to $2\un_{\Sph^1}$. However the classes of $h_1$ and $h_2$ in $\KS$ are different. Actually,  applying the virtual Poincar\'e polynomial \cite{MCPvirt}, we know that the classes of $\Sph^1 \cup \Sph^1$ and $\Sph^1$ are different in $\K$, respectively equal to $2(u+1)$ and $u+1$.
\end{ex}

\section{Duality and link}\label{sect-dual}

We define duality operators on the Grothendieck rings studied in section \ref{sect-Gro}, which are consistent with the duality at the level of constructible functions. We study its properties, notably with respect to the motivic Milnor fibres with sign (by analogy with Bittner's result in \cite{B2}). We investigate also the link operator, defining an analogue of the local link operator in the Grothendieck ring of arc-symmetric sets.

\subsection{Duality}
We begin with the non relative setting where we dispose of a nice description of the Grothendieck ring of $\AS$-sets. First, note that the presentation of the Grothendieck ring of varieties given in Theorem \ref{bit} enables to construct a duality on $\M$.

\begin{thm}(\cite{B}, Corollary 3.4) There exists an involution $\D$ of $\M$ sending $\LL$ to $\LL^{-1}$ and characterised by the property that $\D([X])=\LL^{-\dim X}[X]$ for regular compact varieties $X$.
\end{thm}

The involution $\D$ is called the duality map.
The description of the Grothendick ring of arc-symmetric sets given in Theorem \ref{thm-AS} enables also to define a duality at the level of $\MAS$, duality which is compatible with $\D$. In the following, we identify $\MAS$ with $\Z[u,u^{-1}]$, so that the class of the affine line in $\MAS$ is given by $u$.

\begin{prop} There exists an involution $\DAS:\MAS\to \MAS$ satisfying $\D([X])=u^{-\dim X}[X]$ for compact nonsingular arc-symmetric sets $X$. Moreover the following diagram
$$\xymatrix{\M \ar[d]_{\D}\ar[r]^{\alpha} & \MAS \ar[d]_{\DAS}\ar[r] & \KSA \ar[d]_{\id}\\
\M \ar[r]_{\alpha} & \MAS \ar[r] &  \KSA}$$
is commutative
\end{prop}

\begin{proof}
We define the duality $\DAS$ on $\MAS\simeq \Z[u,1/u]$ as the involution given by $u\mapsto 1/u$. It satisfies the relation $\DAS([X])=u^{-\dim X}[X]$ for a compact nonsingular arc-symmetric set $X$ by Remark \ref{rem-surj}.(2), so that it is compatible with $\D$. The induced morphism on $\KSA\simeq \Z$ is the identity since the Euler characteristic of an odd dimensional smooth compact semialgebraic set is zero, and therefore it gives back the duality for constructible functions on a point, accordingly to Remark \ref{rem-point}.
\end{proof}

\vskip 5mm
 
In the relative setting, the description of $\KS$ in terms of classes of regular varieties $X$ proper over $S$, subject to a blowing-up relation, enables to define a duality involution relative to $S$: there exist a morphism $\DS:\KS\to \MS$ sending the class $[h:X\to S]$ of a regular variety $X$ proper over $S$, to $\LL^{-\dim X}[h:X\to S]$. For $A\in \K$ and $B\in \KS$, it satisfies $\DS(AB)=\D(A)\DS(B)$ so that it can be extended to a $\D$-linear morphism $\DS:\MS\to \MS$. The same construction can be performed at the level of $\MNS$.

As a consequence: 

\begin{thm} 
\begin{enumerate}
\item Let $S$ be a real algebraic variety. There exists an involution $\DS:\MS\to \MS$ which is a $\D$-linear morphism, sending 
the class $[h:X\to S]$ of a regular variety $X$ proper over $S$, to $\LL^{-\dim X}[h:X\to S]$.
\item Let $S$ be a Nash set. There exists an involution $\DNS:\MNS\to \MNS$ which is a $\D$-linear morphism, sending the class $[h:X\to S]$ of a Nash manifold $X$ proper over $S$, to $u^{-\dim X}[h:X\to S]$.
\item  Let $S$ be a real algebraic variety. The diagram 
$$\xymatrix{\MS \ar[d]_{\DS}\ar[r] & \MNS \ar[d]_{\DNS} \ar[r] &  \KSAS \ar[d]_{D_S}\\
\MS \ar[r] & \MNS \ar[r]   & \KSAS}$$
is commutative.
\end{enumerate}
\end{thm}

\begin{proof}
The proof of (1) and (2) are completely similar to the proof of the existence of the duality operator in \cite{B}. They are based on the presentation of the corresponding Grothendieck rings in terms of regular varieties $X$ proper over $S$, subject to a blowing-up relation, together with the following formula (\cite{B}, Lemma 3.5) :
$$[(h\circ \pi)_{|E}: E\to S]=(1+\LL+\cdots+\LL^{d-1})[h_{|C}:C\to S],$$
where $h:X\to S$ is a proper mapping and $X$ is nonsingular, $\pi: Bl_C X\to X$ denotes the blowing-up of $X$ along a nonsingular subvariety $C$ with exceptional divisor $E$.
The commutativity of the diagram comes from the construction of the dualities together with Remark \ref{rem-alg}.
\end{proof}

\begin{rem} We do not know whether one can construct a duality at the level of the relative Grothendieck ring of $\AS$-sets (apart if the morphism $\KNS \to \KASS$ mentioned in Remark \ref{rem-blN}.(\ref{QblN}) happens to be an isomorphism).
\end{rem}

From now on, we concentrate on the duality on $\MS$. The duality commutes with proper push-forward.

\begin{lem}\label{lem-com} Let $f:S\to T$ be a proper regular mapping between real algebraic varieties. Then $f_! \circ \DS = \DT \circ f_!$.
\end{lem}

\begin{proof}
Using the presentation given in Theorem \ref{bit}, we have 
$$\DT \circ f_! [h:X\to S]=\DT [f\circ h:X\to T]=\LL^{-\dim X} [f\circ h:X\to T]$$
since $X$ is nonsingular and $f\circ h$ is proper. On the other hand
$$f_! \circ \DS [h:X\to S]=f_! (\LL^{-\dim X} [h:X\to S])$$
since $h$ is proper and therefore
$$f_! \circ \DS [h:X\to S]=\LL^{-\dim X}f_! [h:X\to S]=\LL^{-\dim X}[f\circ h:X\to T]$$
as required.
\end{proof}

\begin{rem}\label{rem-6} For a general mapping $f:S\to T$, one may define as in \cite{B}, in the spirit of Grothendieck six operations, a map $f_*= \D_{T} \circ f_! \circ \DS$. In general, the duality operator does not commute with pull-back, so that one can define similarly $f^{!}:=\DS \circ f^*\circ \D_{T}$.
\end{rem}


\subsection{Action of the duality on the motivic Milnor fibres}
We consider the action of the duality of the motivic Milnor fibres with sign.
We begin with computing the image under the duality of the motivic Milnor fibres with sign computed in Example \ref{ex-mot}.

\begin{ex}\label{ex-dual}\begin{flushleft}\end{flushleft}
\begin{enumerate} 
\item Let $f:\mathbb R^2 \longrightarrow
    \mathbb R$ be defined by $f(x,y)=x^2+y^4$. Then $\psi_f^+=\LL+1\in \M$ so that $\D \psi_f^+=\LL^{-1}+1=\LL^{-1}\psi_f^+$.
\item \label{ex3} Let $f:\mathbb R^2 \longrightarrow
    \mathbb R$ be defined by $f(x,y)=y^2+x^2(x^2-1)$. Then 
$$\psi_f^+=[E_{1}\to X_0]+[E_{2}\to X_0]-(\LL+1)[ E_{1,2}\to X_0].$$
Note that $D_{X_0}[E_{1}\to X_0]=\LL^{-1}[E_{1}\to X_0]$, $\D_{X_0}[E_{2}\to X_0]=\LL^{-1}[E_{2}\to X_0]$ and $\D_{X_0}[E_{1,2}\to X_0]=[E_{1,2}\to X_0]$ so that
$$\D_{X_0}\psi_f^+=\D_{X_0}[E_{1}\to X_0]+\D_{X_0}[E_{2}\to X_0]-\D(\LL+1)\D_{X_0}[E_{1,2}\to X_0]=\LL^{-1}\psi_f^+.$$
\end{enumerate}
\end{ex}

In fact that behaviour is general and the following result, a real version of Theorem 6.1 in \cite{B2}, shows that the motivic Milnor fibres with sign behaves as the class of a proper nonsingular object with respect with the duality in $\M_{X_0}$. 

\begin{thm}\label{dualM} The equality $\D_{X_0}\psi_f^+= \LL^{1-d}\psi_f^+$ holds in $\M_{X_0}$.
\end{thm}

\begin{rem} Consider the specialisation of Example \ref{ex-dual}.(\ref{ex3})  at the origin via the morphism $\M_{X_0}\to \M$ induced by the inclusion of the origin $O$ of $\R^2$ in $X_0$. The image $\psi_{f,O}^+$ of $\psi_f^+$ is then equal to
$$\psi_{f,O}^+=2+(\LL+1)-(\LL+1)2=1-\LL\in \M.$$
Note that the relation $\D\psi_{f,O}^+=\LL^{-1}\psi_{f,O}^+$ no longer holds, accordingly to Remark \ref{rem-6}.
\end{rem}

For the proof of Theorem \ref{dualM}, we use the expression of the motivic Milnor fibres with sign given by a resolution of singularities, as in Corollary \ref{cor-fib}. Recall that we know how to compute the duality on a proper mapping defined on a nonsingular variety. Then, if the terms $\tE0I$ appearing in Corollary \ref{cor-fib} happen to be nonsingular, note however that the projection $\tE0I \to X_0$ is not proper in general. Therefore, the strategy of the proof will consist in compactifying $\tE0I \to X_0$ in such a way that we can still compute the duality. It is done in \cite{B2} using a reduction to a toric situation.
The proof  of Theorem \ref{dualM} can be obtained from that in \cite{B2} by way of complexification. We sketch it anyway since it leads to an interesting intermediate result for certain toric varieties. 

\begin{proof}
Consider $I$ such that $\tE0I$ is not empty, and define $\tEI$ to be the normalisation of $E_I$ in $\tE0I$. In particular $\tEI$ is normal and is isomorphic to $\tE0I$ over $E_I^0\subset E_I$. A key point in the proof of Theorem \ref{dualM} is that over $E_J$, with $I\subset J$, the set $\tEI$ is isomorphic to $\tEJ$ (cf. Lemma 5.2 in \cite{B2}). It enables in particular to obtain the following formula for $\psi_f^{\pm}$, using the additivity in $\M_{X_0}$ :
$$\psi_f^{+}=\sum_{I\neq \emptyset} (-1)^{|I|-1}[\PP^{|I|-1}][\tEI\to X_0].$$
Afterwards, it suffices to prove that the proper maps $\tEI\to X_0$ behave under the duality as if $\tEI$ were nonsingular. 

To prove both facts, consider an affine chart $U$ where $f$ is given by $f\circ \sigma=ux_1^{m_1}\cdots x_k^{m_k}$, with $k>0$ and $m_i>0$, where $x_1,\ldots,x_n$ are local coordinates and $u$ is a unit, and suppose that $I$ is given by $I=\{1,\ldots,l\}$, with $l\leq k$. Then $\tE0I$ is isomorphic to $\{(x,s)\in E_I\times \R^*: s^{m_I}=u_I(x)\}$ so that $\tEI$ corresponds to the normalisation of $\{(x,s)\in E_I\times \R: s^{m_I}=u_I(x)\}$.

Consider in the following discussion a connected component of $X$ where $u$ is positive. Actually if $u$ is negative on that component, either $m_I$ or one of the $m_i$, for $i\in I$, is odd and then changing the corresponding coordinate by its opposite makes $u$ positive, or $m_I$ and all $m_i$, for $i\in I$, are even and then $\tE0I$ is empty over that connected component if $u$ is negative.

Now, using the fact that $\tE0I$ is the real part of the normalisation of the complexification of $E_I$ in the complexification of $\tE0I$, we can follow the proof of Lemma 5.2 in \cite{B2} to reduce the problem to the case $u=1$. The situation being toric, we can follow further the proof in \cite{B2}, which leads us to toric result stated below (Lemma 4.1 in \cite{B2}, when an additional group action is considered).
\end{proof}

\begin{lem} Let $X$ be an affine toric variety associated with a simplicial cone, equipped with a proper mapping over a real algebraic variety $S$. Then $\DS [X\to S]=\LL^{-\dim X}[X\to S]$.
\end{lem}

\begin{rem} As a corollary, we obtain (Corollary 4.2 in \cite{B2}) that when $S$ is a toric variety associated with a simplicial fan, then $\DS[S\to S]=\LL^{-\dim S}[S\to S]$. In particular if $S$ is compact $\D [S]=\LL^{-\dim S}[S]$. Passing to the virtual Poincar\'e polynomial, it shows that the coefficients of $u^i$ and $u^{\dim S-i}$ in $\beta(S)$ coincide, similarly to the case when $S$ is compact and nonsingular. 
\end{rem}


\subsection{Link}
We use the duality $\DS$ to define a group morphism $\LS:\MS \to \MS$ by the formula $\LS=1+\LL\DS$. In particular, if $h:X\to S$ is proper and $X$ nonsingular, then we obtain
$$\LS ([h:X\to S])=(1+\LL^{1-\dim X}) [h:X\to S]=\D([\Sph^{\dim X-1}])[h:X\to S],$$
generalising the formula in Remark \ref{rem-alg}.
This makes $\LS$ a generalisation of the topological link, defined on semialgebraically constructible functions, at the level of $\MS$, as shown in next proposition. However, we will see later that this definition is not compatible with the virtual Poincar\'e polynomial, giving a motivation for Theorem \ref{thm-l} below.

\begin{prop}\label{prop-link} We have a commutative diagram 
$$\xymatrix{\MS \ar[d]_{\LS}\ar[r] & \KSAS  \ar[d]_{\LSAS}\\
\MS \ar[r] &  \KSAS}$$
\end{prop}

\begin{proof}
Let $h:X\to S$ be a proper mapping from a nonsingular algebraic variety $X$. Then 
$$\LS([h:X\to S]=(1+\LL^{1-\dim X})[h:X\to S],$$
so that its image in $\KSAS$ is given by $(1+(-1)^{1-\dim X})[h:X\to S]$ because the class of the affine line in $\KSA$ is equal to $-1$. We conclude using Remark \ref{rem-alg}.
\end{proof}

The map $\LS$ on $\MS$ satisfies similar properties as the topological link on constructible functions, as in Proposition \ref{prop-const}.(1), (2) and (3).

\begin{lem}\label{lem-sim}\begin{enumerate}
\item Let $f:S\to T$ be a proper regular mapping between real algebraic varieties. Then $f_! \circ \LS = \lambda_T \circ f_!$.
\item $\LS\circ \LS=2\LS$.
\item $\LS \circ \DS=\LL \DS \circ \LS$.
\end{enumerate} 
\end{lem}

\begin{proof} For the first part, apply Lemma \ref{lem-com} and the $\M$-linearity of $f_!$. The second one comes from the $\D$-linearity of $\DS$. Actually 
$$\DS\circ (\LL \DS)=\LL^{-1} \DS \circ \DS=\LL^{-1}\id$$
so that
$$\LS\circ \LS=1+2\LL\DS+\LL \DS(\LL \DS)=2\LS.$$
The proof of the last point is similar.
\end{proof}

We can identify the image of $\LS$ in $\MS$.
Denote by $\ES$ the subgroup of $\MS$ generated by $(1+\LL^{1-\dim X})[h:X\to S]$, with $X$ nonsingular and $h$ proper. 

\begin{prop}\label{prop-LS} For $k\in \mathbb N$, the image of the $k$-th iterate $\LS^k$ of $\LS$ is included in $2^{k-1}\ES$.
\end{prop}
\begin{proof} For $k=1$, it suffices to compute $\LS$ on $[h:X\to S]$ with $X$ nonsingular and $h$ proper, and in that case
$$\LS [h:X\to S]=(1+\LL^{1-\dim X})[h:X\to S].$$
For general $k$, we use Lemma \ref{lem-sim}.
\end{proof}

Proposition \ref{prop-LS} enables to state an analogue of Proposition \ref{prop-const}.(4) in the motivic setting. Denote by $\E$ the subgroup of $\M$ generated by $(1+\LL^{1-\dim X})[X]$ for compact nonsingular $X\in \R\Var$. Note that the image of $\E$ under the Euler characteristic with compact supports $\chi_c:\M\to \Z$ is equal to zero since the Euler characteristic of an odd dimensional smooth compact semialgebraic set is equal to zero.

Denote by $\int_S:\M_S\to \M$ the $\M$-linear mapping induced by the push-forward onto a point.

\begin{cor} Assume $S$ is compact. The image of $\int_S \LS:\MS \to \M$ is included in $\E$.
\end{cor}

\begin{proof}
Compute $\int_S \LS$ on $[h:X\to S]$ with $X$ nonsingular and $h$ proper. Then
$$\int_S \LS [h:X\to S]=(1+\LL^{1-\dim X})[X].$$
Note that $X$ is compact because so is $S$ and $h$ is proper.
\end{proof}

Semialgebraically constructible functions whose link takes only even values are called Euler \cite{MCP}. Using Proposition \ref{prop-LS}, we recover the fact that algebraically constructible function are Euler.

\begin{cor}(\cite{MCP}) Algebraically constructible function are Euler.
\end{cor}

\begin{proof}
We know by Theorem \ref{prop-alg} that any algebraically constructible function is the image of an element in $\MS$. Moreover, using Proposition \ref{prop-LS}, it suffices to consider the image of $\ES$ in $\KSAS$. Finally, the value of $1+\LL^{1-d}$ is equal to $0$ or $2$ when one replaces $\LL$ with $-1$. 
\end{proof}


\subsection{Local link and Virtual Poincar\'e polynomial}

The virtual Poincar\'e polynomial of the link of an algebraic variety at a point is well-defined by \cite{FS}. Even more, the link, as an algebraic set, is well-defined up to $\AS$-homeomorphism and does not depend on the embedding of the algebraic variety in the ambient smooth space (by Proposition 7.4 in \cite{PP}).

Let $S\subset \R^n$ be an algebraic set, and fix a point $s\in S$. We consider the link of $S$ at $s$ defined using the distance function to $s$ in $\R^n$ given by the square of the Euclidean distance to $s$.
We are going to define the analogue of $\Lambda_S \phi (s)$, where $\phi \in F(S)$, at the level of $\KS$, replacing the Euler characteristic with compact supports with the virtual Poincar\'e polynomial. 

\begin{rem}
\begin{enumerate}
\item Combining Remark \ref{rem-alg} with Proposition \ref{prop-fub}, we have 
$$\chi_c(h^{-1}(\lk(s,S)))=\chi_c(\Sph^{\dim X -1}) \chi_c(h^{-1}(s))$$ 
if $h:X\to S$ is a proper semialgebraic mapping defined on a nonsingular algebraic variety $X$.
\item However, if $h:X\to S$ is a proper regular mapping on a nonsingular algebraic variety $X$, we have
$$\beta(h^{-1}(\lk(s,S))) \neq \beta(\Sph^{\dim X -1}) \beta(h^{-1}(s))$$
in general as illustrated by Example \ref{ex-torus} below. 
\end{enumerate}
\end{rem}

\begin{ex}\label{ex-torus} Consider the classical height function $h:T\to \R$ on a torus $T$. Then $h$ admits four singular values corresponding to the minimal and maximal values $s_-$ and $s_+$, and the values $s_1$ and $s_2$ corresponding to the two saddle points, with $s_- < s_1 < s_2 < s_+$. Then 
\begin{displaymath} \beta(h^{-1}(\lk(s,\R)))=
\left\{ \begin{array}{lllll}
0 \textrm{~~if~~} s\notin [s_-,s_+],\\
u+1 \textrm{~~if~~} s\in \{s_-,s_+\},\\
2(u+1) \textrm{~~if~~} s\in (s_-,s_1)\cup (s_2,s_+),\\
3(u+1) \textrm{~~if~~} s\in \{s_1,s_2\},\\
4(u+1) \textrm{~~if~~} s\in (s_1,s_2).\\
\end{array} \right.
\end{displaymath}
whereas 
\begin{displaymath} \beta(h^{-1}(s))=
\left\{ \begin{array}{lllll}
0 \textrm{~~if~~} s\notin [s_-,s_+],\\
1 \textrm{~~if~~} s\in \{s_-,s_+\},\\
u+1 \textrm{~~if~~} s\in (s_-,s_1)\cup (s_2,s_+),\\
u \textrm{~~if~~} s\in \{s_1,s_2\},\\
2(u+1) \textrm{~~if~~} s\in (s_1,s_2).\\
\end{array} \right.
\end{displaymath}
This shows that $\Ls$ is not directly related to $\LS$ in general, due to the different behaviours between the virtual Poincar\'e polynomial and the Euler characteristic with compact supports with respect to integration.
\end{ex}

However, we can define a local link operator in the following sense, which is a motivic counterpart for Proposition \ref{prop-fub}.

\begin{thm}\label{thm-l} There exist a linear mapping $\Ls:\KS \to \KAS$ sending the class of $h:X\to S$ in $\KS$ to the class of $h^{-1}(\lk(s,S))$ in $\KAS$. In particular
$$\beta \circ  \Ls [h:X\to S]=\beta(h^{-1}(\lk(s,S)))$$ 
in $\Z[u]$.
\end{thm}

\begin{proof}
Let $h:X\to S$ be a regular mapping. Then $h^{-1}(\Sph(s,\epsilon))$ is $\AS$-homeomorphic to $h^{-1}(\Sph(s,\epsilon'))$ for $\epsilon, \epsilon'$ small enough by Corollary 9.7 in \cite{PP}, therefore the class of $h^{-1}(\lk(s,S))$ in $\KAS$ is well-defined. It just depends on the class $[h:X\to S]$ of $h:X\to S$ in $\KS$ because a regular isomorphism is in particular an $\AS$-homeomorphism. Finally if $Y\subset X$ is a closed subvariety, then the class in $\KAS$ of $h^{-1}(\Sph(s,\epsilon))$, $h_Y^{-1}(\Sph(s,\epsilon))$ and $h_{X\setminus Y}^{-1}(\Sph(s,\epsilon))$ are well-defined for $\epsilon$ small enough, and then
$$[h^{-1}(\lk(s,S))]=[h_Y^{-1}(\lk(s,S))]+[h_{X\setminus Y}^{-1}(\lk(s,S))]$$
in $\KAS$, so that $\Ls$ is additive.
\end{proof}

\begin{rem} For a subvariety $X\subset S$, we obtain in particular
$$\Ls([X\to S])=[\lk(s,X)]\in \KAS.$$
\end{rem}
 
 

It makes sense to compare the link $\Ls$ in $\KAS$ of a non-negative function having an isolated zero, with the image in $\MAS$ of the local positive motivic Milnor at that point. The next result says that both classes coincide in the case of a weighted homogeneous polynomial vanishing only at the origin.

\begin{prop} Let $f\in \R[x_1,\ldots,x_d]$ be a non-negative weighted homogeneous polynomial vanishing only at the origin in $\R^d$. Then 
$$\beta(\psi_{f,0}^+)=\Ls [f:\R^d \to \R]=1+u^{d-1}\in \MAS.$$
\end{prop}

\begin{proof}
Note that $f$ is necessarily non-degenerate with respect to its Newton polyhedron, so that one can apply Proposition \ref{propM} to obtain $\psi_{f,0}^+=[\{f=1\}]$ in $\M$. By homogeneity of $f$, the set $\{f=1\}$ is isomorphic to $f^{-1}(\eps)$ for positive $\epsilon\in \R$, so that $\psi_{f,0}^+=f^{-1}(\eps)$ in $\M$. Note moreover that $f^{-1}(\eps)$ is a nonsingular algebraic set homeomorphic to a $(d-1)$-dimensional sphere, so that its virtual Poincar\'e polynomial is equal to $1+u^{d-1}$ by Theorem \ref{thm-virt}.
\end{proof}

\begin{rem}
However the link $\Ls$ and the local positive motivic Milnor do not coincide in general, as illustrated by Example \ref{ex-mot}.(\ref{ex2}) where the local positive motivic Milnor vanishes.
\end{rem}


\enddocument